\documentclass[10pt,reqno]{amsart}
\usepackage{epsfig,amssymb,amsmath,version}
\usepackage{amssymb,version,graphicx,fancybox,mathrsfs,multirow}
\usepackage{epstopdf}
\usepackage{url,hyperref}
\usepackage{bm}
\usepackage{subfigure}
\usepackage{color,xcolor}
\usepackage{cases}
\usepackage{mathtools}
\usepackage{extarrows}
\usepackage{bbm,bm}
\usepackage[ruled,vlined]{algorithm2e}
\usepackage{algorithmic}
\usepackage{wrapfig}
  \usepackage{makecell}

\allowdisplaybreaks[4]
\catcode`\@=11 \theoremstyle{plain}
\@addtoreset{equation}{section}   
\renewcommand{\theequation}{\arabic{section}.\arabic{equation}}
\@addtoreset{figure}{section}
\renewcommand\thefigure{\thesection.\@arabic\c@figure}
\renewcommand{\thefigure}{\arabic{section}.\arabic{figure}}
\newtheorem{thm}{\bf Theorem}

\newtheorem{cor}{\bf Corollary}
\newtheorem{prop}{Proposition}[section]

\newtheorem{lmm}{\bf Lemma}
\theoremstyle{remark}

\newenvironment{lemma}{\begin{lmm}}{\end{lmm}}

\theoremstyle{remark}
\newtheorem{rem}{\bf Remark}[section]
\theoremstyle{definition}
\newtheorem{defn}{\bf Definition}[section]

\numberwithin{table}{section}

\def \ri {{\rm i}}

\newcommand{\refe}[1]{{\rm (\ref{#1})}}

\newcommand{\bs}[1]{\boldsymbol{#1}}
\def \hf {\hat{f}}

\def \d {{\rm d}}

\def \e {{\rm e}}
\def \ri {{\rm i}}
\def \bx {{\bs x}}
\def \hx {\hat{\bs x}}
\def \bk {{\bs k}}
\def \by {{\bs y}}

\def \bk {{\bs k}}

\def \bn {{\bs n}}
\def \bmm {{\bs m}}

\def \bss {{\bs s}}

\def \RR {{\mathbb R}}

\begin{document}
\bibliographystyle{plain}
 \baselineskip 13pt
	
 \title[FEM for Nonlocal Laplacian]
 {An Efficient Finite Element Method for Multi-dimensional Nonlocal Laplacian on Uniform Grids}
   \author[C. Sheng, \;  H. Li, \; L.-L. Wang,\;\;$\&$\;H. Yuan]{ 
		 \;\; Changtao Sheng${}^\dag$,\;\; Huiyuan Li${}^\ddag$, \;\; Huifang Yuan${}^{\star}$\;\;\; and\;\;  Li-Lian Wang${}^{\P}$}
	
	\thanks{${}^\dag$School of Mathematics, Shanghai University of Finance and Economics, Shanghai, 200433, China. Email: ctsheng@sufe.edu.cn (C. Sheng). The work of this author is partially supported by the National Natural Science Foundation of China (Nos. 12201385 and 12271365). \\
                \indent  ${}^{\ddag}$State Key Laboratory of Computer Science/Laboratory of Parallel Computing, Institute of Software,
 Chinese Academy of Sciences, Beijing 100190, China. The work of this author is partially supported by the National Natural Science Foundation of China (No. 12131005 and 11971016).  Email: huiyuan@iscas.ac.cn (H. Li).\\
                  \indent ${}^{\star}$Corresponding Author. School of Science, Harbin Institute of Technology, Shenzhen, Guangdong, China. The work of the  author is supported by research grants: 12301463 from NSFC, HA11409085, a start-up Grant from Shenzhen government and HA45001129, a start-up Grant from Harbin Institute of Technology, Shenzhen. Email:  yuanhuifang@hit.edu.cn (H. Yuan).  \\
                  \indent ${}^{\P}$Division of Mathematical Sciences, School of Physical and Mathematical Sciences, Nanyang Technological University, 637371, Singapore. The research of the  author is partially supported by Singapore  MOE AcRF Tier 1 Grant: RG95/23, and Singapore MOE AcRF Tier 2 Grant: MOE-T2EP20224-0012. Email: lilian@ntu.edu.sg (L.-L. Wang). 
         }
\begin{abstract} Computing the stiffness matrix for the finite element discretization of the nonlocal Laplacian on unstructured meshes is difficult, because the operator is nonlocal and can even be singular.   
In this paper, we focus on the $C^0$-piecewise linear finite element method (FEM) for the nonlocal Laplacian on uniform grids within a $d$-dimensional rectangular domain. 
By leveraging the connection between FE bases and  B-splines (having attractive convolution properties), we can reduce the involved $2d$-dimensional integrals for the stiffness matrix entries into integrations over $d$-dimensional balls with explicit integrands involving cubic B-splines and the kernel functions, which allows for  explicit study of the  singularities and accurate evaluations of such integrals in spherical coordinates. We show the nonlocal stiffness matrix has a block-Toeplitz structure, so the matrix-vector multiplication can be implemented  using fast Fourier transform (FFT).  In addition, when the interaction radius $\delta\to 0^+,$  the nonlocal stiffness matrix automatically reduces to the local one.  Although our semi-analytic approach on uniform grids cannot be extended to general domains with unstructured meshes, the resulting solver can seamlessly integrate with the grid-overlay (Go) technique for the nonlocal Laplacian on arbitrary bounded domains.   
\end{abstract}
 
 \subjclass[2000]{34B10, 65R20,  15B05, 41A25, 74S05.}	
\keywords{Nonlocal Laplacian, finite element method, B-splines, convolution, nonlocal stiffness matrix, semi-analytic method}

\maketitle

 \vspace*{-15pt}

\section{Introduction}
This paper concerns the finite element solution  of the following nonlocal constrained value problem posed on a $d$-dimensional rectangular domain $\Omega$:
\begin{equation}\label{1DNon}
\mathcal{L}_{\delta} u = f\;\;{\rm on}\;\;  \Omega\,;\quad
u = 0\;\; {\rm on}\;\; \Omega_{\delta},
\end{equation}
where $
\Omega_{\delta}=\left\{\bs x \in \mathbb{R}^{d} \backslash \Omega\,:\, \text {dist}(\bs x, \partial \Omega) \leq \delta\right\}$ and the nonlocal operator is defined as
\begin{equation}\label{NonLOper} 
\mathcal{L}_{\delta} u(\bx) =\int_{\mathbb{B}^d_\delta} (u(\bx+\bss)-u(\bx))\rho_{\delta}(|\bss|)\, \d \bss,
\end{equation}
with $\mathbb{B}^d_\delta$ being a $d$-dimensional ball of radius $\delta$ centered at $0$. 
We further assume the radial type kernel $\rho_{\delta}(|\bss|)$ is a non-increasing function with a bounded second-order moment:
\begin{equation}\label{ker2}
\int_{\mathbb{B}_\delta^d}\rho_{\delta}(|\bss|)|\bss|^2\d \bss=d,
\end{equation}
which particularly includes the fractional power kernel:
\begin{equation}\label{kernel}
\rho_\delta(|\bss|)=c_{d,\delta}^\alpha |\bss|^{-d-\alpha},\;\;\;
\end{equation}
where $c^\alpha_{d,\delta}$ is a normalization constant that satisfies \eqref{ker2}. 
Such nonlocal models with more general kernels in general domains  have many applications in e.g., solid mechanics, anomalous diffusion, image processing, and material failure (see e.g., \cite{andreu2010nonlocal,buades2010image,Du2012Analysis,Gilboa2007Nonlocal,Gilboa2008Nonlocal,gunzburger2010nonlocal,kao2010random,klafter2005anomalous,metzler2000random}). Indeed, modeling  differentiations and variations with integral operators can accommodate discontinuities such as fractures and material separation. A representative example is peridynamics \cite{Silling2000Reformulation}, a derivative-free integral continuum theory effective for crack nucleation, growth, and failure \cite{du2011mathematical,kilic2010coupling,oterkus2012peridynamic,silling2010peridynamic,silling2010crack}. Nonlocality also captures long-range interactions that classical local models often miss, especially in multiscale materials bridging molecular dynamics and classical elasticity \cite{askari2008peridynamics,Du2019Nonlocal}.    
 
 These benefits come at the expense of substantial computational and memory costs, thereby prompting a strong demand for developing  efficient numerical methods.
Successful attempts include quadrature-based finite-difference methods in  one dimension or a two-dimensional rectangular domain \cite{Tian2013analysis,TianJuDu2017}, which led to the notion of asymptotically compatible schemes. 
  We  refer to \cite{Chen2011Continuous,Wang2012,Tian2014Asymptotically,leng2022a} (for finite element methods in 1D), 
\cite{TianJuDu2015,Pasetto2022,Glusa2023an} (for smooth kernel on 2D rectangular domain),  \cite{Vollmann2019} (using adaptive Gauss-Kronrod quadrature rule for hypersingular kernel on 3D rectangular domain), spectral methods for nonlocal problems with periodic boundary conditions \cite{Du2016Asymptotically,Du2017fast} and on the sphere \cite{Slevinsky2018}. It is also noteworthy some other numerical approaches   \cite{Wang2014a,Du2020a,Yang2022On,Du2022Perfectly,Leng2021Asymptotically,Wang2024Stability} mostly on regular domains.
Nevertheless, the implementation of FEM on unstructured meshes for irregular domains becomes much more complicated \cite{DElia2021,Du2022on,chen2024}. To illustrate the intrinsic  difficulty, we consider \eqref{1DNon}-\eqref{kernel} in a $d$-dimensional bounded domain $\Omega=\Lambda$ with a triangulation $\mathcal T_\Lambda$ consisting of $K_\Lambda$ non-overlapping  elements $\{\mathcal E_k\}_{k=1}^{K_\Lambda}$. Let $\{\varphi_j\}$ be the FEM basis functions associated with $\mathcal T_\Lambda$. The underlying   stiffness matrix  requires to compute
 \begin{equation}\label{eq:Ah-global}
\begin{aligned}
A_\delta(\varphi_j,\varphi_{j'})
&\approx \sum_{k=1}^{K_\Lambda}\int_{\mathcal E_k}\int_{\Lambda\cap B_{\delta,h}(\bx)}
\big(\varphi_j(\by)-\varphi_j(\bx)\big)\big(\varphi_{j'}(\by)-\varphi_{j'}(\bx)\big)\,
\rho(\bx,\by)\, \d\by\, \d\bx \\
&\quad + 2\sum_{k=1}^{K_\Lambda}\int_{\mathcal E_k}
\varphi_j(\bx)\,\varphi_{j'}(\bx)\bigg(\int_{\Lambda_\delta\cap B_{\delta,h}(\bx)} \rho(\bx,\by)\, \d\by\bigg)\, \d\bx,
\end{aligned}
\end{equation}
where $B_{\delta,h}(\bx)$ denotes a polytopal approximation of the ball $B_{\delta}(\bx)$ of radius $\delta$ centered at $\bx$, and $\Lambda_{\delta}=\{\bs x \in \mathbb{R}^{d} \backslash \Lambda:\, \mathrm{dist}(\bs x, \partial \Lambda) \le \delta\}$. In 2D, D'Elia et al.\ \cite{DElia2021} developed and analyzed ball-approximation strategies to accurately evaluate integrals \eqref{eq:Ah-global} over the intersection between the interaction ball and a mesh element. 
The treatment of singular kernels necessitates sophisticated numerical quadrature, which makes implementation more demanding and may introduce additional computational cost. Subsequently, Du et al.\ \cite{Du2022on} study polygonal approximations of Euclidean interaction balls and analyze their impact on nonlocal operators with smooth kernels and solutions in the small/vanishing-horizon regime. To address the practical difficulties that certain ball-approximation strategies pose in 3D, Chen et al.\ \cite{chen2024} devised a combinatorial-map–based FEM for smooth kernels, which enables efficient $\delta$-neighborhood queries to evaluate the inner integrals in \eqref{eq:Ah-global} over $\Lambda\cap B_{\delta,h}(\bx)$, robust ball approximations, and fast matrix assembly for nonlocal operators. However, for hypersingular kernels $\rho(\bx,\by)$, more efficient high-order quadrature on unstructured meshes within the $\delta$-neighborhood is still required to accurately evaluate the corresponding hypersingular integrals. Moreover, each stiffness-matrix entry relies on the aforementioned $2d$-dimensional integral and thus incurs a significant computational burden.

\begin{wrapfigure}[11]{r}{0.3\textwidth} 
  \vspace{-6pt}                           
  \centering
  \includegraphics[width=0.28\textwidth]{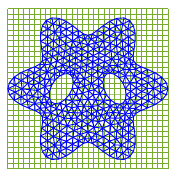}
  
  {\small Illustration of Go  in \cite{huang2024}}
  \label{fig:go}
\end{wrapfigure}

We intend to propose a different approach to tackle FEM for nonlocal Laplacian  on unstructured meshes by employing the grid-overlay (Go) technique proposed in \cite{huang2024,huang2025} 
(for the fractional Laplacian on arbitrary bounded domains, see the figure on the right). The key building block of this framework is a solver for the nonlocal problem on background uniform grids that enclose the irregular computational domain, leading to stiffness matrix that enables fast matrix–vector multiplications.   
 
Consider for example the FEM solution of  \eqref{1DNon}-\eqref{kernel}  on a two-dimensional bounded irregular  domain $\Lambda$.  We outline  the essential steps to integrate the FEM on uniform grids with the Go technique, dubbed as Go-FEM, as follows.
\begin{itemize}
  \item[(i)]  Triangulate $\Lambda$ into an unstructured simplicial mesh $\mathcal{T}_\Lambda$ that fits or approximately fits $\Lambda$ (see the blue mesh).  Overlay it with a background uniform  mesh   $\mathcal{R}_\Omega$ (see the green mesh) of a rectangular domain $\Omega\supset \Lambda.$

  \smallskip
\item[(ii)]  Construct the transfer matrix ${\mathbb T}_\Lambda^{\Omega}$ that maps data from the unstructured mesh $\mathcal {T}_\Lambda$ to the background grid $\mathcal{R}_{\Omega}$, e.g., the piecewise linear interpolation in \cite[(3.3)]{huang2024}.

\smallskip  
  \item[(iii)]  Compute the nonlocal stiffness matrix $\bs S_\Omega$ on the uniform grids  $\mathcal{R}_{\Omega},$ which should allow for fast matrix-vector multiplication.   
  
  \smallskip
  \item[(iv)] Form the Go-FEM solver as in \cite[(1.2)]{huang2024} on   $\mathcal {T}_\Lambda$:
\begin{equation}\label{GoFEM}
  \mathbb {D}_\Lambda^{-1} (\mathbb {I}_\Lambda^{\Omega})^\top\, \bs{S}_\Omega\, {\mathbb T}_\Lambda^{\Omega}\, {\bs u}_\Lambda = \bs f_\Lambda,
  \end{equation}
where $\bs {u}_\Lambda$ denotes the numerical solution on $\Lambda$, and $\mathbb D_\Lambda$ is diagonal with entries given by the column sums of $\mathbb T_\Lambda^{\Omega}$, which ensures a consistent, constant-preserving transfer.
\end{itemize}
\smallskip

\noindent In \cite{huang2024,huang2025}, Step-(iii) involves a finite difference solver for fractional Laplacian, and  the proposed Go-FD can take full advantages of both uniform-grid solvers in efficient matrix-vector multiplications and unstructured meshes for complex geometries and mesh adaptation.  Along this line, it is  essential to compute the nonlocal stiffness matrix on uniform grids. This serves as the purpose this paper and we shall report the full Go-FEM for nonlocal models in the future work \cite{ShenSH2025}.

We show in this paper that by leveraging the convolution structure of B-spline basis functions, the original $2d$-dimensional integrals in the bilinear form (see \eqref{BilinearF} below) can be reduced to $d$-dimensional integrals. 
More specifically, consider a uniform partition of $\Omega$ with mesh size $h_j$ along the $x_j$-direction and $1\le j\le d$, and define the $C^0$ piecewise linear tensor-product finite element nodal basis functions as $\{ \phi_{\bn}(\bx) = \prod_{j=1}^{d} \phi_{n_j}(x_j)\}_{1 \le n_j \le N_j}.$ 
Our main finding is that the entries of the stiffness matrix  can be expressed as
\begin{equation} \label{intro1}\begin{split}
(\bs S_\delta)_{ll^\prime} &
= \frac12 \int_{\Omega\cup\Omega_\delta}\int_{\mathbb{B}^d_\delta}(\phi_\bn(\bx+\bss)-\phi_\bn(\bx))(\phi_\bmm(\bx+\bss)-\phi_\bmm(\bx))\rho_{\delta}(|\bss|)\,\d \bss\,\d \bx
\\&= \frac{h^d}{2} \int_{\mathbb{B}^d_\delta} f_{\bk}(\bss) \, \rho_{\delta}(|\bss|) \, \mathrm{d}\bss,\;\;\;\bs k = (k_1, \ldots, k_d),\;\;\;k_j = |n_j - m_j|,
\end{split}\end{equation}
where the integrand $f_{\bk}(\bss)$ is a linear combination of B-spline functions (see Theorem \ref{Toepp3d}).
This result effectively reduces the computational complexity inherent in the Galerkin-type weak formulation to the level of strong-form discretizations, thereby greatly simplifying the overall algorithmic structure.  Moreover, the integrand $f_\bk(\bx)$ has a zero of multiplicity two at the origin (see Proposition \,\ref{Smulti}), which makes it computationally tractable and suitable for both integrable and hypersingular kernels. Similar strategy has already been applied to fractional Laplacian equations \cite{sheng2021fast}, where, in two dimensions, the fourfold integral reduces to a one-dimensional integral. 
Second, for $\delta > 0$, the resulting nonlocal FEM stiffness matrix is a block-Toeplitz matrix with sparse Toeplitz blocks, whose bandwidth increases with $\delta$. This structure enables fast matrix-vector multiplications via FFT-based algorithms.
Third, for a uniform mesh with element size $h_j\equiv h$, when $\delta \leq h$, the entries of the generating matrix associated with the hypersingular kernel~\eqref{kernel} in two and three dimensions can be computed exactly. This result reveals that, from the perspective of the stiffness matrix, the nonlocal formulation asymptotically recovers the classical stiffness matrix for Laplace operators as $\delta \to 0$. To the best of our knowledge, this is a new observation not previously reported in the literature. In the case $\delta > h$, the singularity of the hypersingular kernel near the origin can be exploited to split the integral into singular and regular parts for more efficient computation.

The rest of this paper is organized as follows. In section \ref{sec2}, we derive the explicit integral representation for the entries of the stiffness matrix for nonlocal model \eqref{1DNon} in arbitrary dimensions, and exploit the Toeplitz structure of the matrix and the two limiting cases with respect to the parameter $\delta$. The detailed implementation for calculation of the entries matrix  is introduced in section \ref{sec3}.
In section \ref{sec5}, ample numerical results for nonlocal problems with both regular and hypersingular kernels, including accuracy testing and limiting behaviors are presented to demonstrate the efficiency and accuracy. We conclude the paper with some remarks.

\section{Q1-FEM for nonlocal Laplacian on rectangular meshes}\label{sec2}
\setcounter{equation}{0}
\setcounter{lmm}{0}
\setcounter{thm}{0}

In this section, we introduce the several ingredients we use to construct the efficient finite element method, including the B-spline functions, the generalized convolution property of the B-spline functions, and $d$-dimensional spherical coordinates. The explicit integral expressions of the entries of the nonlocal FEM stiffness matrix, its block Toeplitz structure, as well as the two limiting cases of the nonlocal stiffness matrix when the interaction radius $\delta \rightarrow 0$ and $\delta \rightarrow \infty$ are discussed in this section, while the complete implementation for the computation of each element of the stiffness matrix are presented in the next sections.

Let us first introduce some notations. Let $\mathbb{R}$ (resp. $\mathbb{N}$) be the set of all real numbers (resp. positive integers), and
let $\mathbb{N}_0=\mathbb{N}\cup \{0\}$.  We use boldface lowercase letters to denote $d$-dimensional multi-indexes, vectors and multi-variables,
e.g., $ \bk=(k_1,\cdots,k_d)$  and $\bx=(x_1,\cdots,x_d)$.
Also, let $e_i=(0,\cdots,1,\cdots,0)$ be the $i$-th unit vector in $\mathbb{R}^d$, and denote 
\begin{equation}\label{notationvector}
|\bk|_1=\sum^d_{j=1}k_j,\;\;\;|\bk|_{\infty}=\max_{1\leq j\leq d}k_j.
\end{equation}

 \subsection{FEM setting}\label{femset}

Denote $\widetilde  \Omega:=\Omega\cup\Omega_\delta,$ and define the natural energy space associated with \eqref{1DNon} as
\begin{equation}\label{FESpace}
\mathcal H^1_{0,\delta}(\widetilde  \Omega) =\bigg\{u\in L^2(\widetilde  \Omega)\, :\, \int_{\widetilde\Omega}\int_{\mathbb{B}^d_\delta}(u(\bx+\bss)-u(\bx))^2\rho_{\delta}(|\bss|)\,\d \bss\,\d \bx < \infty,  \,u|_{\Omega_{\delta}}=0\bigg\},
\end{equation}
as a counterpart  of the usual $H^1_0$-space for local problems. 
A weak formulation of \eqref{1DNon} is to find $u\in \mathcal H^1_{0,\delta}(\widetilde  \Omega) $ such that 
\begin{equation}\label{WeakF}
{\mathcal A}_\delta(u,v) = (f,v)_{\Omega},\quad \forall v\in \mathcal H^1_{0,\delta}(\widetilde  \Omega),
\end{equation}
where $(\cdot,\cdot)_{\Omega}$ is the usual $L^2$-inner product on $\Omega,$ and 
\begin{equation}\label{BilinearF}
\begin{split}
 & {\mathcal A}_\delta(u,v) = \frac12 \int_{\widetilde\Omega}\int_{\mathbb{B}^d_\delta}(u(\bx+\bss)-u(\bx))(v(\bx+\bss)-v(\bx))\rho_{\delta}(|\bss|)\,\d \bss\,\d \bx.
\end{split}
\end{equation}

Let $\Omega=(a_1,b_1)\times\cdots\times(a_d,b_d)$ be a $d$-dimensional rectangular cuboid with each of the $d$ sides having a potentially different length. Accordingly,  
we set
$\widetilde \Omega=(a_1-\delta,b_1+\delta)\times\cdots\times(a_d-\delta,b_d+\delta),$ and have $\Omega_\delta=\widetilde \Omega\setminus \Omega.$
We uniformly partition the domain $\Omega$ along each coordinate direction  into
\begin{equation}\label{meshmn}
\begin{split}
\mathcal{T}_h=\big\{(\xi_{n_1},\xi_{n_2},\cdots,\xi_{n_d})\,:\, 0\leq n_j\leq N_j+1\;\text{with}\;\;1\leq j\leq d\big\},\;\;
\end{split}\end{equation}
where $\xi_{n_j}:=\xi_{j,n_j}=a_j+ n_j h_j$ with $h_j=\frac{b_j-a_j}{N_j+1}$. We introduce the usual 
piecewise linear FEM basis functions:   
\begin{equation}
\label{0piecewiselinearbasis}
\phi_{n_j}(x_j)=\begin{dcases}
\frac{x_j-\xi_{n_{j}-1}}{h_j},\quad & {\rm if}\;\; x_j\in (\xi_{n_{j}-1},\xi_{n_j}),\\[3pt]
\frac{\xi_{n_{j}+1}-x_j}{h_j}, \quad & {\rm if}\;\; x_j\in (\xi_{n_j},\xi_{n_j+1}),\\[3pt]
0,\quad & \text{elsewhere on}\;\; [a_j-\delta, b_j+\delta],
\end{dcases}
\end{equation}
and define the $d$-dimensional approximation space 
\begin{equation}
\label{appspacerd}
{\mathcal V}^0_{h,\delta}:={\rm span}\Big\{\phi_{\bn}(\bx)=\prod^d_{j=1}\phi_{n_j}(x_j)\,:\, \bn\in\Upsilon_{N}\Big\}, 
\end{equation}
where the index set
\begin{equation}\label{UpsilonA}
\Upsilon_{N}:=\big\{\bn=(n_1,\cdots, n_d)\,:\, 1\le n_j\le N_j,\; 1\le j\le d\big\}.
\end{equation}  
The FEM for \eqref{WeakF} is to find  
$u_{h}\in \mathcal V^{0}_{h,\delta}$ such that 
\begin{equation}\label{uvshx}\begin{split}
{\mathcal A}_\delta(u_h,v_h)
=(f,v_h)_{\Omega},\quad \forall\, v_h\in \mathcal V^{0}_{h,\delta}.
\end{split}
\end{equation}
 It is known that both \eqref{WeakF} and \eqref{uvshx} are well-posed by the standard Lax-Milgram Lemma.

  Given this setting, we intend to explore the best possible analytic information for accurately  computing
 the entries of the stiffness matrix.  
For this purpose, we write
\begin{equation}\label{uexpansion}
u_h(\bx)=\sum_{\bn\in \Upsilon_{\!N}} \tilde{u}_{\bn}\phi_{\bn}(\bx)\in {\mathcal V}^0_{h,\delta},\quad \forall\, \bx\in \widetilde \Omega.
\end{equation}
Substituting  \eqref{uexpansion} into \eqref{uvshx}, we can obtain the following linear system
\begin{equation}\label{dDsystem}
\bs S_\delta\, \tilde {\bs u}={\bs f},
\end{equation}
where we arrange the 
$N_\ast=\prod^d_{j=1}N_j$ unknowns in column-major order, i.e., $\bs u$ is a column vector in $\mathbb R^{N_*},$ and can form the stiffness matrix $\bs S_\delta$ and the vector $\bs f$ accordingly. For instance, in 2D case, we arrange the unknowns in column-major order, that is,
\begin{equation*}\label{ufform}
\begin{split}
&\tilde{\bs u}=\big(\underbrace{\tilde{u}_{11},\tilde{u}_{21},\cdots,\tilde{u}_{N_11}}_{n_2=1}, \, \underbrace{\tilde{u}_{12},\tilde{u}_{22},\cdots,\tilde{u}_{N_12}}_{n_2=2},\cdots,
\underbrace{\tilde{u}_{1N_2},\tilde{u}_{2N_2},\cdots,\tilde{u}_{N_1N_2}}_{n_2=N_2}\big)^{\top}\in {\mathbb R}^{N_*}, 
\end{split}
\end{equation*}
where $\tilde u_{n_1n_2}$ is the $l$-th element of $\tilde{\bs u}$ with $l=(n_2-1)N_1+n_1.$ Correspondingly,  the nonlocal stiffness matrix $\bs S_\delta$ is an $N_1N_2$-by-$N_1N_2$ symmetric matrix with  the entries
\begin{equation}\label{2DSdelta}
(\bs S_\delta)_{ll'}={\mathcal A}_\delta(\phi_{n_1}\phi_{n_2},\phi_{m_1}\phi_{m_2}),\quad l=(n_2-1)N_1+n_1, \;  l'=(m_2-1)N_1+m_1,
\end{equation}
and 
\begin{equation*} 
\bs {f}=(f_1,\cdots,f_{N_\ast})^\top,\quad f_{l^\prime}=(f,\phi_{m_1}\phi_{m_2}),\quad l'=(m_2-1)N_1+m_1.
\end{equation*}
Similarly, in the 3D case, we have
\begin{equation*}
\tilde {\bs u}=\big(\underbrace{\tilde{u}_{111},\tilde{u}_{211},\cdots,\tilde{u}_{N_1 1 1}}_{n_2=n_3=1}, \, \underbrace{\tilde{u}_{112},\tilde{u}_{212},\cdots,\tilde{u}_{N_1 1 2}}_{n_2=1,n_3=2},\cdots,
\underbrace{\tilde{u}_{1 N_2 N_3},\tilde{u}_{2 N_2 N_3},\cdots,\tilde{u}_{N_1 N_2 N_3}}_{n_2=N_2,n_3=N_3}\big)^{\top} \in \mathbb{R}^{N_*},  
\end{equation*}
and the coefficient $\tilde u_{n_1n_2n_3}$ in \eqref{uexpansion} is located at the position of  $l=(n_3-1)N_1N_2+(n_2-1)N_1+n_1$ in the unknown vector $\tilde {\bs u},$
and the 
corresponding nonlocal stiffness matrix $\bs S_\delta\in {\mathbb R^{N_*\times N_*}}$ is a symmetric matrix with the entries
\begin{equation}\label{3DSdelta}
(\bs S_\delta)_{ll'}={\mathcal A}_\delta(\phi_{n_1}\phi_{n_2}\phi_{n_3},\phi_{m_1}\phi_{m_2}\phi_{m_3}),\quad l'=(m_3-1)N_1N_2+(m_2-1)N_1+m_1,
\end{equation}
where the global index $l$ is the same as above. Similarly,
\begin{equation*} 
\bs {f}=(f_1,\cdots,f_{N_\ast})^\top,\quad f_{l^\prime}=(f,\phi_{m_1}\phi_{m_2}\phi_{m_3}).
\end{equation*}
We can form the system in 
\eqref{dDsystem} in the same manner in any dimension.

\subsection{B-splines and  some useful properties}
In the computation of the nonlocal stiffness matrix, we mainly use the relation between finite element bases and  B-splines,  where  the convolution properties of the B-splines are essential for the derivation. 

It is known that the cardinal B-splines can be defined recursively by convolution (see e.g., \cite{Chui1992introduction,Milovanovi2010}).

\begin{defn}[{\bf Cardinal B-splines}]\label{Bspline}{\em 
The cardinal B-spline of degree $p\in \mathbb{N}$ is defined via convolution
\begin{equation}\label{bp}
B_{p}(t)=\big(B_{p-1}\ast B_0\big)(t)=\int_{\RR}B_{p-1}(t-s)B_0(s)\,\d s=\int_{0}^1B_{p-1}(t-s)\,\d s,\quad p\ge 1,
\end{equation}
where the cardinal B-spline of degree $0$ is the  characteristic function of the interval $[0,1),$ i.e., 
$B_0(t)=1$ for $t\in [0,1)$ and vanishes elsewhere on $\mathbb R.$}
\end{defn}

We have the following important properties.
\begin{lemma}\label{propertyB}
The cardinal B-spline $B_p(t)$ of degree $p\in \mathbb{N}$ has the properties as follows.
\smallskip
\begin{itemize}
    \item[(i)]  ${\rm sup}(B_{p})\subseteq I_p:=[0,p+1]$, and its a piecewise polynomial of degree $p-1$ on each subinterval $[k,k+1]$ with $k=0,\cdots p.$ Globally, $B_p\in C^{p-1}(\mathbb R).$
    \medskip
    \item [(ii)] For  $t\in I_p$ and $p\ge 2,$
    \begin{equation}\label{relatAB-S} 
\begin{aligned}
B_p(t) =\frac{t}{p-1} B_{p-1}(t)+\frac{p-t}{p-1} B_{p-1}(t-1);  \quad 
B_p^{\prime}(t)=B_{p-1}(t)-B_{p-1}(t-1).
\end{aligned}
\end{equation}
\item[(iii)] It is symmetric on the interval $I_p,$ i.e., 
\begin{equation}\label{symeqA} 
B_{p}\Big(\frac{p+1}2+t\Big)=B_{p}\Big(\frac{p+1}2-t\Big)\quad {\rm or}\quad B_{p}(t)=B_{p}(p+1-t).
\end{equation}
\item[(iv)] For $p,q\ge 0,$
\begin{equation}\label{innereq}\begin{split}
\int_{\RR}\!B_{p}(t+s)B_{q}(s)\,\d s=B_{p+q+1}(p+1-t)=B_{p+q+1}(q+1+t). 
\end{split}\end{equation}
\end{itemize}
\end{lemma}
\begin{proof}
The proofs of the first three properties can be found in \cite{Chui1992introduction,Milovanovi2010}. Here, we sketch the proof of the convolution formula (iv), which will be mainly used for the derivation of the nonlocal stiffness matrix.  Let $\widehat B_p(\omega)$ be the Fourier transform of $B_p(t).$ Then by the definition \eqref{bp}, we have 
$\widehat B_p(\omega)=\big(\widehat B_0(\omega)\big)^{p+1}$ and  $\widehat B_{p+q+1}(\omega)= \widehat B_{p}(\omega)\cdot \widehat B_{q}(\omega)$
for $p,q\ge 0.$ Thus, applying Fourier inverse transform, we have   
\begin{equation*}\label{innereqold}\begin{split}
B_{p+q+1}(t)=(B_p*B_{q})(t)=\!\int_{\RR}B_{p}(t-s)B_{q}(s)\,\d s
=\int_{\RR}B_{p}(p+1-t+s)B_{q}(s)\,\d s,
\end{split}\end{equation*}
where in the last step, we used \eqref{symeqA} for $B_p.$ Substituting $t$ by  $p+1-t,$ we obtain the first identity of \eqref{innereq}, and the second identity is a direct consequence of     \eqref{symeqA}.
\end{proof}

 It is known that the finite element basis in \eqref{0piecewiselinearbasis} 
 can be expressed in $B_1,$ i.e.,
\begin{equation}\label{p1basisrelation}
\tilde{\phi}_{n_j}(x_j)=B_1\Big(\frac{x_j-\xi_{n_{j}-1}}{h_j}\Big), \quad B_1(t)=\begin{dcases}
t,\;\;&t\in[0,1),\\
2-t,\;\;&t\in[1,2),\\
0,\;\;&\text{elsewhere,}
\end{dcases}
\end{equation}
where $\tilde{\phi}_{n_j}$ is the zero extension of ${\phi}_{n_j}$  to $\mathbb R.$  In the late part,  we also need to use the cubic spline given by 
\begin{equation}\label{bspline4}
B_3(t)=\begin{dcases}
\frac{t^3}{6},\;\;&t\in[0,1),\\
-\frac{t^3}2+2t^2-2t+\frac23,\;\;&t\in[1,2),\\
\frac{t^3}2-4t^2+10t-\frac{22}3,\;\;&t\in[2,3),\\
-\frac{(t-4)^3}{6},\;\;&t\in[3,4),\\
0,\;\;&\text{elsewhere.}
\end{dcases}
\end{equation}
  
\subsection{Nonlocal FEM stiffness matrix} In what follow, we restrict our attention to high dimensional cube with $h_1=\cdots =h_d=h$.  
We have the following integral representations of the entries of the nonlocal stiffness matrix.
\begin{thm}\label{Toepp3d}   Let $\{\phi_{\bs n}(\bs x)\}$ be the tensorial FEM basis defined in 
\eqref{0piecewiselinearbasis}-\eqref{appspacerd}, and $B_3(t)$ be the cubic splines given in \eqref{bspline4}. Then 
the entries of the  nonlocal stiffness matrix $\bs S_\delta \in {\mathbb R}^{N_*\times N_*}$ with $N_*=\prod_{j=1}^dN_j$,
are given by
\begin{equation}\label{tkjdd}
 (\bs S_\delta)_{ll^\prime}
={\mathcal A}_\delta(\phi_\bn,\phi_\bmm)=\frac{h^d}{2}\int_{\mathbb{B}^d_\delta} f_{\bk}(\bss) \,\rho_{\delta}(|\bss|)\,\d \bss, \end{equation} 
where 
\begin{equation}\label{cgdefqdd}  \begin{split}
f_{\bk}(\bss)=2\prod^d_{j=1}B_3(k_j+2)- \prod^d_{j=1}B_3(k_j+2-s_j/h)-\prod^d_{j=1}B_3(k_j+2+s_j/h),
\end{split}
\end{equation}
with  $\bs k=(k_1,\cdots, k_d), k_j=|n_j-m_j|,$ and the local-to-global indices: 
\begin{equation}\label{indexrelad}\begin{split}
l= \sum_{i=1}^{d} (n_i - 1) \Big(\prod_{j=1}^{i-1} N_j\Big) + 1,
\quad  l^\prime= \sum_{i=1}^{d} (m_i - 1) \Big(\prod_{j=1}^{i-1} N_j\Big) + 1.
\end{split}\end{equation}
\end{thm}
\begin{proof}
Let $\tilde{\phi}_{\bn}(\bx)$ be the zero extension of $\phi_{\bn}(\bx)$ on $\widetilde  \Omega=\Omega\cup\Omega_\delta$   to $\mathbb{R}^d$, and  likewise for $\tilde{\phi}_{\bs m}(\bx).$
Then we derive from \eqref{3DSdelta} that
\begin{equation}\label{BilinearF2x3d}
\begin{split}
&{\mathcal A}_\delta(\phi_{\bn},\phi_{\bmm})=\frac12\int_{\mathbb{B}^d_\delta} \int_{\widetilde\Omega}\big(\phi_\bn(\bx+\bss)-\phi_\bn(\bx)\big)\big(\phi_\bmm(\bx+\bss)-\phi_\bmm(\bx)\big) \,\d \bx\,\rho_{\delta}(|\bss|)\,\d \bss 
\\&= \frac12\int_{\mathbb{B}^d_\delta} \int_{\RR^d}\big(2\tilde{\phi}_\bn(\bx) \tilde{\phi}_\bmm(\bx)- \tilde{\phi}_\bn(\bx+\bss) \tilde{\phi}_\bmm(\bx)- \tilde{\phi}_\bn(\bx)\tilde{\phi}_\bmm(\bx+\bss)\big) \,\d \bx\,\rho_{\delta}(|\bss|)\,\d \bss  
\\&= \frac12\int_{\mathbb{B}^d_\delta} \big(g_1-g_2(\bss)-g_3(\bss)\big) \,\rho_{\delta}(|\bss|)\,\d \bss,
\end{split}
\end{equation}
where 
\begin{equation} \label{fg123}\begin{split}
&g_1=2\!\int_{\RR^d}\tilde{\phi}_{\bn}(\bx)\tilde{\phi}_{\bmm}(\bx) \,\d \bx=2\prod_{j=1}^d\int_{\RR}\tilde{\phi}_{n_j}(x_j)\tilde{\phi}_{m_j}(x_j) \,\d x_j,\\
& g_2(\bss)=\!\int_{\RR^d}\tilde{\phi}_{\bn}(\bx+\bss)\tilde{\phi}_{\bmm}(\bx) \,\d \bx=\prod_{j=1}^d\int_{\RR}\tilde{\phi}_{n_j}(x_j+s_j)\tilde{\phi}_{m_j}(x_j) \,\d x_j,\\
& g_3(\bss)=\!\int_{\RR^d}\tilde{\phi}_{\bn}(\bx)\tilde{\phi}_{\bmm}(\bx+\bss)\,\d \bx =\prod_{j=1}^d\int_{\RR}\tilde{\phi}_{n_j}(x_j)\tilde{\phi}_{m_j}(x_j+s_j)\,\d x_j.
\end{split}\end{equation}
Using the change of variable $y=\tfrac{x_j-\xi_{n_j-1}}{h}$, we find from \eqref{p1basisrelation} readily that
\begin{equation}\begin{split} \label{f1d2d_x1} 
 \int_{\RR}\tilde{\phi}_{n_j}(x_j) \tilde{\phi}_{m_j}(x_j) \,\d x_j&=\int_{\RR}B_1\Big(\frac{x_j-\xi_{n_j-1}}{h}\Big) B_1\Big(\frac{x_j-\xi_{m_j-1}}{h}\Big) \,\d x_j
\\&=h\int_{\RR}B_1(y) B_1\Big(y+\frac{\xi_{n_j-1}-\xi_{m_j-1}}{h}\Big) \,\d y
\\&=h\int_{\RR}B_1(y) B_1\big(y+n_j-m_j\big) \,\d y\\[6pt]
&=h\, B_3(n_j-m_j+2)=h\, B_3(m_j-n_j+2),
\end{split}\end{equation}
where we used the property \eqref{innereq}  including  $B_3(2+t)=B_3(2-t)$
to derive the last  two identities.  Similarly, with the same change of variable, we obtain from \eqref{innereq}-\eqref{p1basisrelation} that
 \begin{equation} \label{f2d2d_x1}\begin{split}
\int_{\RR}\tilde{\phi}_{n_j}(x_j+s_j) &\tilde{\phi}_{m_j}(x_j) \,\d x_j=\int_{\RR}B_1\Big(\frac{x_j+s_j-\xi_{n_j-1}}{h}\Big) B_1\Big(\frac{x_j-\xi_{m_j-1}}{h}\Big) \,\d x_j 
\\&= h\int_{\RR}B_1\Big(y+\frac{s_j+\xi_{m_j-1}-\xi_{n_j-1}}{h}\Big)B_1(y)  \,\d y
\\[4pt]&=h B_3\big(m_j-n_j+2+s_j/h\big)=h B_3\big(n_j-m_j+2-s_j/h\big),
\end{split}\end{equation}
and 
 \begin{equation} \label{f3d2d_x1}\begin{split}
\int_{\RR}\tilde{\phi}_{n_j}(x_j)&\tilde{\phi}_{m_j}(x_j+s_j)\,\d x_j  =\int_{\RR}\tilde{\phi}_{n_j}(x_j-s_j)\tilde{\phi}_{m_j}(x_j)\,\d x_j\\
&=hB_3\big(n_j-m_j+2+s_j/h\big)=hB_3\big(m_j-n_j+2-s_j/h\big).
\end{split}\end{equation}
Thus, we derive from \eqref{fg123}-\eqref{f3d2d_x1} that
\begin{equation*}\begin{split}
g_1-g_2(\bss)-g_3(\bss)=h^d\Big(2\prod^d_{j=1}B_3(k_j+2)- \prod^d_{j=1}B_3(k_j+2-s_j/h)-\prod^d_{j=1}B_3(k_j+2+s_j/h)\Big),
\end{split}\end{equation*}
where 
$k_j=|n_j-m_j|.$ Denoting  
$ f_{\bs k}(\bs s):=(g_1-g_2(\bs s)-g_3(\bs s))/h^d$,
we derive  \eqref{tkjdd}-\eqref{cgdefqdd} from
\eqref{BilinearF2x3d} and the above immediately.
\end{proof}

It is seen from the above theorem that (i) the computation of the entries reduces  to integrations over the ball of interaction; 
and (ii)  by  
\eqref{bspline4}, $B_3$ has a support $[0, 4),$ so the sparsity and bandwidth  of $\bs S$ depends on $\delta.$ We next illustrate  this   and provide more explicit formulas in two and three dimensions. 
\subsubsection{Two-dimensional case}
With the ordering in \eqref{2DSdelta}, we can obtain the following block Toeplitz structure of the stiffness matrix along with expressions for their entries.
\begin{cor}[{\bf 2D nonlocal stiffness matrix}]
    \label{Toepsdel} 
If $\delta>0$, the nonlocal FEM stiffness matrix $\bs S_\delta$ is an $N_2$-by-$N_2$ block-Toeplitz matrix of the form 
\begin{equation}\label{asdel_2d}
\bs S_\delta= 
\begin{bmatrix}
   \boldsymbol T_{0} &\hspace{-6pt} \boldsymbol T_{1} &   \cdots &   \boldsymbol T_{N_2-2} & \;\; \boldsymbol T_{N_2-1} \\[1pt]
   \boldsymbol T_{1}  &\hspace{-6pt} \boldsymbol T_{0}   & \ddots   &\ddots &\boldsymbol T_{N_2-2}\\[0pt]
      \vdots& \hspace{-6pt}\ddots  & \ddots   & \ddots  & \vdots \\[-1pt]
   \boldsymbol T_{N_2-2}   & \ddots    & \ddots   & \boldsymbol T_{0} &\boldsymbol T_{1} \\[5pt]
   \boldsymbol T_{N_2-1}   & \;\; \boldsymbol T_{N_2-2}     & \cdots\hspace{-2pt}     & \boldsymbol T_{1}  & \boldsymbol T_{0}
  \end{bmatrix},
  \end{equation} 
where each block of $\bs S_\delta$  is  an $N_1$-by-$N_1$ Toeplitz matrix  and  the  entries are given by
\smallskip
\begin{equation}\label{tkj1_2d}
\bs T_{k_2}=\begin{bmatrix}
   t_{0}^{k_2} &\hspace{-4pt} t_{1}^{k_2} & \cdots  & t_{N_1-2}^{k_2} & t_{N_1-1}^{k_2} \\[1pt]
   t_{1}^{k_2} &\hspace{-4pt} t_{0}^{k_2}   & \ddots & \ddots  & t_{N_1-2}^{k_2}\\[0pt]
      \vdots& \hspace{-4pt}\ddots  & \ddots    & \ddots & \vdots \\[-1pt]
   t_{N_1-2}^{k_2}  & \ddots   & \ddots   & t_{0}^{k_2} & t_{1}^{k_2} \\[5pt]
   t_{N_1-1}^{k_2}  & t_{N_1-2}^{k_2}   &\cdots \hspace{-2pt}    & t_{1}^{k_2} & t_{0}^{k_2}
  \end{bmatrix},    
\end{equation}
for $0\leq k_2\leq N_2-1.$
Here, for each $k_2,$ the entries in the first row of the generating matrix $\bs T_{k_2}$ can be evaluated in polar coordinates  by
\begin{equation}\label{gpa00_2d} 
t_{k_1}^{k_2}= h^2\int_0^{\pi} \Big(\int_{0}^{\delta} r \rho_{\delta}(r)\, \hat{f}_{k_1,k_2}(r,\theta)\, \d r \Big) \,\d\theta, \quad 0\le k_1\le N_1-1,
\end{equation}
where 
\begin{equation}\label{cgdefq2d} 
\begin{split}
\hat{f}_{k_1,k_2}(r,\theta)&= 2B_3(k_1+2)B_3(k_2+2)-B_3(k_1+2+r\cos\theta/h)B_3(k_2+2+ r\sin\theta/h)
\\[1pt] &\quad -B_3(k_1+2-r\cos\theta/h)B_3(k_2+2- r\sin\theta/h). 
\end{split}
\end{equation}
\end{cor}
\begin{proof}
The formulas in  
\eqref{tkjdd}- \eqref{cgdefqdd} with $d=2$ reduce to
\begin{equation}\label{stiffmatr2d}
\begin{split}
(\bs S_\delta)_{ll'}={\mathcal A}_\delta(\phi_{n_1}\phi_{n_2},\phi_{m_1}\phi_{m_2}) & =\frac{h^2}{2}\int_{\mathbb{B}^2_\delta} f_{k_1,k_2}(\bss) \,\rho_{\delta}(|\bss|)\,\d \bss,
\end{split}
\end{equation}
where 
\begin{equation}\label{fin2d}  \begin{split}
f_{k_1,k_2}(\bss)=&\,2 B_3(k_1+2)B_3(k_2+2)-B_3(k_1+2-s_1/h)B_3(k_2+2-s_2/h)
\\&-B_3(k_1+2+s_1/h)B_3(k_2+2+s_2/h),
\end{split}
\end{equation}
with the local-to-global indices $l=(n_2-1)N_1+n_1$ and  $l'=(m_2-1)N_1+m_1$. 
We continue the evaluation using  polar coordinates, and using  \eqref{stiffmatr2d} and the above leads to 
\begin{equation}\label{combinationa1}
\begin{split}
(\bs S_\delta)_{ll'}={\mathcal A}_\delta(\phi_{n_1}\phi_{n_2},\phi_{m_1}\phi_{m_2})&=
\frac{h^2}2\int_0^\delta
\int_{-\pi}^{\pi}f_{k_1,k_2}(r\cos\theta,r\sin\theta)\, r \rho_{\delta}(r)\,\d \theta\,\d r
\\&=
h^2\int_0^\delta
\int_{0}^{\pi}f_{k_1,k_2}(r\cos\theta,r\sin\theta)\, r \rho_{\delta}(r)\,\d \theta\,\d r
\\&=
h^2\int_0^\delta
\int_{0}^{\pi}\hat{f}_{k_1,k_2}(r,\theta)\, r \rho_{\delta}(r)\,\d \theta\,\d r:=t^{k_2}_{k_1},
\end{split}
\end{equation}
where $\hat{f}(r,\theta)$ is defined in \eqref{cgdefq2d} and we used the fact that the integrand is even in $\theta$.

 It remains to show that the nonlocal stiffness matrix  $\bs S_\delta$ with this ordering leads to the block Toeplitz matrix in \eqref{asdel_2d}-\eqref{tkj1_2d}.  
To this end, we denote the entries of the $N_1N_2$-by-$N_1N_2$ stiffness matrix $\bs S_\delta$ in terms of $n_1,n_2,m_1,m_2$ by  $s^{n_1,n_2}_{m_1,m_2}.$ In view  of $s^{n_1,n_2}_{m_1,m_2}=(\bs S_\delta)_{ll'}$
and the definitions of $l,l'$,
we have 
\begin{equation*} 
\bs S_{\delta}=\begin{bmatrix}
  s_{1,1}^{1,1} & s_{1,1}^{2,1} & \cdots  & s_{1,1}^{N_1,1} & \cdots  & s_{1,1}^{1,N_2} &s_{1,1}^{2,N_2} & \cdots & s_{1,1}^{N_1,N_2}\\[6pt]
  s_{2,1}^{1,1} & s_{2,1}^{2,1} & \cdots  & s_{2,1}^{N_1,1} & \cdots  & s_{2,1}^{1,N_2} & s_{2,1}^{2,N_2} & \cdots & s_{2,1}^{N_1,N_2}\\[6pt]
  \vdots & \vdots & \ddots & \vdots &   & \vdots & \vdots &  \ddots & \vdots\\[6pt]
  s_{N_1,1}^{1,1} & s_{N_1,1}^{2,1} & \cdots  & s_{N_1,1}^{N_1,1} & \cdots  & s_{N_1,1}^{1,N_2} &  s_{N_1,1}^{2,N_2} & \cdots & s_{N_1,1}^{N_1,N_2}\\[6pt]
  \vdots & \vdots &   & \vdots &  \ddots & \vdots & \vdots &  & \vdots\\[6pt]
  s_{1,N_2}^{1,1} & s_{1,N_2}^{2,1} & \cdots  & s_{1,N_2}^{N_1,1} & \cdots  & s_{1,N_2}^{1,N_2} & s_{1,N_2}^{2,N_2} & \cdots & s_{1,N_2}^{N_1,N_2}\\[6pt]
  s_{2,N_2}^{1,1} & s_{2,N_2}^{2,1} & \cdots  & s_{2,N_2}^{N_1,1} & \cdots  & s_{2,N_2}^{1,N_2} & s_{2,N_2}^{2,N_2} & \cdots & s_{2,N_2}^{N_1,N_2}\\[6pt]
   \vdots & \vdots &   & \vdots &   & \vdots & \vdots &\ddots  & \vdots\\[6pt]
   s_{N_1,N_2}^{1,1} & s_{N_1,N_2}^{2,1} & \cdots  & s_{N_1,N_2}^{N_1,1} & \cdots  & s_{N_1,N_2}^{1,N_2} & s_{N_1,N_2}^{2,N_2}& \cdots & s_{N_1,N_2}^{N_1,N_2} 
  \end{bmatrix}.  
\end{equation*}
Then by \eqref{combinationa1}, we have $$s^{n_1,n_2}_{m_1,m_2}=(\bs S_\delta)_{ll'}=t_{k_1}^{k_2}=t_{|n_1-m_1|}^{|n_2-m_2|},$$ which implies  $\bs S_\delta $ is identical to the block Toeplitz matrix in \eqref{asdel_2d}-\eqref{tkj1_2d}. This completes  the proof.
\end{proof}

In view of  the above block Toeplitz structure, the matrix-vector multiplication can be performed via FFT in $O(N_1N_2\log N_1N_2)$ operations.  Next, we present several important properties of the 2D stiffness matrix.
\medskip

\noindent {\bf (i)\, Sparsity of $\bs S_{\delta}.$}\, Since ${\rm support}(B_3)\subseteq [0,4]$,  one verifies readily that 
if $(\max\{j,k\}+2)-\delta/h\ge 4,$ then $B_3(j+2)B_3(k+2)=0,$ and
\begin{equation}\label{sparseprop}
B_3(j+2\pm r\cos\theta/h)B_3(k+2\pm r\sin\theta/h)=0,\;\;\;\forall(r,\theta)\in(0,\delta)\times (0,\pi).
\end{equation}  Thus, for given $0\le k_2\le N_2-1,$ the $N_1$ entries $\{t_{k_1}^{k_2}\}_{k_1=0}^{N_1-1}$ that generate the Teoplitz matrix  $\bs T_{k_2}$ are given by
\begin{equation}\label{gpa00_2dxx} t_{k_1}^{k_2}= \begin{dcases} h^2\int_0^{\pi} \Big(\int_{0}^{\delta} r \rho_{\delta}(r)\,  \hat{f}_{k_1,k_2}(r,\theta)\, \d r \Big) \,\d\theta,\quad & 0\le k_1,k_2< \delta/h+2,\\
 0,\quad & \max\{k_1,k_2\}\ge\delta/h+2, 
\end{dcases}\end{equation}
where $\{\hat{f}_{k_1,k_2}\}$ are defined in \eqref{cgdefq2d}. 
Moreover, there holds the  symmetry property: 
 \begin{equation}\label{sym2d}
 t_{k_1}^{k_2}=t_{k_2}^{k_1},\quad \forall\,k_1,k_2\ge 0.
 \end{equation}
 Consequently,  the Toeplitz matrix $\bs T_k$ is sparse with a  band-width that increases as the interaction radius  $\delta$ increases.

\medskip

\noindent {\bf (ii)\, Limit of $\bs S_{\delta}$ as $\delta\to 0^+.$}\,  
With the above formulas at our disposal, we consider the limiting case $\delta\to 0,$ and show that with the choice of corresponding kernel, we are able to derive the siffness matrix of the usual  Laplacian operator. 
More precisely, choose
\begin{equation}\label{fracker0_2d}
   \rho_{\delta}(|\bss|) = \frac{C_{\delta}^\alpha} 
   {|\bss|^{2+\alpha}},\;\;\;\; \alpha\in [-1,2),
\end{equation}
where we take $C_{\delta}^\alpha ={2(2-\alpha)\delta^{\alpha-2}}/\pi$ 
so that  $\frac{1}{2}\int_{\mathbb{B}^2_\delta} |\bss|^2 \rho_{\delta}(|\bss|) {\d} \bss=1.$
For $0\le \delta\le h,$  we can evaluate \eqref{gpa00_2dxx}  directly and find the non-zero entries as flollows
\begin{eqnarray*} 
&& t_0^0=\frac{\alpha-2}2\Big(\frac{-1}{3\pi(\alpha-6)}\delta^4+\frac{32}{15\pi(\alpha-5)}\delta^3-\frac{1}{\alpha-4}\delta^2-\frac{64}{9\pi(\alpha-3)}\delta\Big)+\frac{8}{3},
\\&&t_1^0= t_0^1=\frac{\alpha-2}2\Big(\frac{2}{9\pi(\alpha-6)}\delta^4-\frac{56}{45\pi(\alpha-5)}\delta^3+\frac{1}{2(\alpha-4)}\delta^2+\frac{40}{27\pi(\alpha-3)}\delta\Big)-\frac{1}{3},
\\&& t_1^1=\frac{\alpha-2}2\Big(\frac{-4}{27\pi(\alpha-6)}\delta^4+\frac{32}{45\pi(\alpha-5)}\delta^3-\frac{1}{4(\alpha-4)}\delta^2+\frac{32}{27\pi(\alpha-3)}\delta\Big)-\frac{1}{3},
\\&& t_2^0=t_0^2=\frac{\alpha-2}2\Big(\frac{-1}{18\pi(\alpha-6)}\delta^4+\frac{8}{45\pi(\alpha-5)}\delta^3-\frac{16}{27\pi(\alpha-3)}\delta\Big),
\\&& t_2^1=t_1^2=\frac{\alpha-2}2\Big(\frac{1}{27\pi(\alpha-6)}\delta^4-\frac{4}{45\pi(\alpha-5)}\delta^3-\frac{4}{27\pi(\alpha-3)}\delta\Big),
\\&& t_2^2=\frac{\alpha-2}2\Big(\frac{1}{108\pi(\alpha-6)}\delta^4\Big),
\end{eqnarray*}
which implies 
\begin{equation}\label{a10_2d}
\lim\limits_{\delta\to 0^+}\bs{S}_{\delta}=\bs S_0:= \bs S_{h}\otimes\bs M_{h}+\bs M_{h} \otimes \bs S_{h},
  \end{equation}
  where
$\bs S_{h}=\frac 1 {h}{\rm diag}(-1,2,-1)$ and $\bs M_{h}=\frac{h}6 {\rm diag}(1,4,1).$ It is identical to the usual FEM stiffness matrix.

\subsubsection{Three-dimensional case} 
In 3D case, with the ordering in \eqref{3DSdelta}, we can obtain the following block Toeplitz structure of the stiffness matrix along with expressions for their entries.
\begin{cor}[{\bf 3D nonlocal stiffness matrix}]
If $\delta>0$, then the nonlocal FEM stiffness matrix $\bs S_\delta$ has the following Toeplitz-structure:
\begin{equation*} 
\bs S_\delta=  \begin{bmatrix}
  \boldsymbol T_{x_1,x_2,0} &\hspace{-6pt} \boldsymbol T_{x_1,x_2,1} &   \cdots &   \boldsymbol T_{x_1,x_2,N_3-2} & \;\; \boldsymbol T_{x_1,x_2,N_3-1} \\[1pt]
  \boldsymbol T_{x_1,x_2,1} &\hspace{-6pt} \boldsymbol T_{x_1,x_2,0}   & \ddots   &\ddots &\boldsymbol T_{x_1,x_2,N_3-2}\\[0pt]
     \vdots& \hspace{-6pt}\ddots  & \ddots   & \ddots  & \vdots \\[-1pt]
  \boldsymbol T_{x_1,x_2,N_3-2}  & \ddots    & \ddots   & \boldsymbol T_{x_1,x_2,2} &\boldsymbol T_{x_1,x_2,1} \\[5pt]
  \boldsymbol T_{x_1,x_2,N_3-1}  & \;\; \boldsymbol T_{x_1,x_2,N_3-2}    & \cdots\hspace{-2pt}     & \boldsymbol T_{x_1,x_2,1} & \boldsymbol T_{x_1,x_2,0}
 \end{bmatrix},
  \end{equation*} 
where for $k_3=0,1,\cdots,N_3-1$,  the block matrix is given by
\begin{equation*} 
\bs T_{x_1,x_2,k_3}=  \begin{bmatrix}
  \boldsymbol T_{x_1,0,k_3} &\hspace{-6pt} \boldsymbol T_{x_1,1,k_3} &   \cdots &   \boldsymbol T_{x_1,N_2-2,k_3} & \;\; \boldsymbol T_{x_1,N_2-1,k_3} \\[1pt]
  \boldsymbol T_{x_1,1,k_3} &\hspace{-6pt} \boldsymbol T_{x_1,0,k_3}   & \ddots   &\ddots &\boldsymbol T_{x_1,N_2-2,k_3}\\[0pt]
     \vdots& \hspace{-6pt}\ddots  & \ddots   & \ddots  & \vdots \\[-1pt]
  \boldsymbol T_{x_1,N_2-2,k_3}  & \ddots    & \ddots   & \boldsymbol T_{x_1,0,k_3} &\boldsymbol T_{x_1,1,k_3} \\[5pt]
  \boldsymbol T_{x_1,N_2-1,k_3}  & \;\; \boldsymbol T_{x_1,N_2-2,k_3}    & \cdots\hspace{-2pt}     & \boldsymbol T_{x_1,1,k_3} & \boldsymbol T_{x_1,0,k_3}
 \end{bmatrix},
  \end{equation*} 
with each block of $\bs T_{x_1,k_2,k_3}$ is an $N_1$-by-$N_1$ symmetric Toeplitz matrix and their entries are
\begin{equation*} 
\bs T_{x_1,k_2,k_3}=  \begin{bmatrix}
  t_{0,k_2,k_3} &\hspace{-6pt} t_{1,k_2,k_3} &   \cdots &  t_{N_1-2,k_2,k_3} & \;\;t_{N_1-1,k_2,k_3} \\[1pt]
  t_{1,k_2,k_3} &\hspace{-6pt} t_{0,k_2,k_3}   & \ddots   &\ddots &t_{N_1-2,k_2,k_3}\\[0pt]
     \vdots& \hspace{-6pt}\ddots  & \ddots   & \ddots  & \vdots \\[-1pt]
  t_{N_1-2,k_2,k_3}  & \ddots    & \ddots   & t_{0,k_2,k_3} &t_{1,k_2,k_3} \\[5pt]
  t_{N_1-1,k_2,k_3}  & \;\; t_{N_1-2,k_2,k_3}    & \cdots\hspace{-2pt}     &t_{1,k_2,k_3} & t_{0,k_2,k_3}
 \end{bmatrix},
  \end{equation*} 
   for $k_2=0,1,\cdots,N_2-1$ and $k_3=0,1,\cdots,N_3-1$. The entries of generating matrix $\bs{G}=(t_{k_1,k_2,k_3})$ can be evaluated in spherical coordinates by
\begin{equation}\label{gpa3d} 
 t_{k_1,k_2,k_3}=  \frac{h^3}2\int_{0}^{\delta}\int_{0}^\pi\int_0^{2\pi}\hat{f}_{k_1,k_2,k_3}(r,\theta,\varphi)\sin(\theta)\,\d\varphi\d\theta \, \rho_{\delta}(r)r^2\, \d r,
 \end{equation}
 where 
\begin{equation}\label{hatfin3d}  \begin{split}
&\hat{f}_{k_1,k_2,k_3}(r,\theta,\varphi)=\,2 B_3(k_1+2)B_3(k_2+2)B_3(k_3+2)\\&-B_3(k_1+2-r\sin\theta\cos\varphi/h)B_3(k_2+2-r\sin\theta\sin\varphi/h)B_3(k_3+2-r\cos\theta/h)
\\&-B_3(k_1+2+r\sin\theta\cos\varphi/h)B_3(k_2+2+r\sin\theta\sin\varphi/h)B_3(k_3+2+r\cos\theta/h).
\end{split}
\end{equation}
\end{cor}
\begin{proof}
Using the general results in  Theorem \ref{Toepp3d} with $d=3$, the formulas  \eqref{tkjdd}-\eqref{cgdefqdd} reduce to
\begin{equation}\label{stiffmatr3d}
\begin{split}
(\bs S_\delta)_{ll'}={\mathcal A}_\delta(\phi_{n_1}\phi_{n_2}\phi_{n_3},\phi_{m_1}\phi_{m_2}\phi_{m_3}) & =\frac{h^3}{2}\int_{\mathbb{B}^3_\delta} f_{k_1,k_2,k_3}(\bss) \,\rho_{\delta}(|\bss|)\,\d \bss,
\end{split}
\end{equation}
where 
\begin{equation}\label{fin3d}  \begin{split}
f_{k_1,k_2,k_3}(\bss)=&\,2 B_3(k_1+2)B_3(k_2+2)B_3(k_3+2)\\&-B_3(k_1+2-s_1/h)B_3(k_2+2-s_2/h)B_3(k_3+2-s_3/h)
\\&-B_3(k_1+2+s_1/h)B_3(k_2+2+s_2/h)B_3(k_3+2+s_3/h),
\end{split}
\end{equation}
and the local-to-global indices $l=(n_3-1)N_1N_2+(n_2-1)N_1+n_1$ and  $l'=(m_3-1)N_1N_2+(m_2-1)N_1+m_1$. 
We continue the evaluation using  spherical coordinates $s_1=r\sin\theta\cos\varphi$, $s_2=r\sin\theta\sin\varphi$, and $s_3=r\cos\theta$, we obtain from  \eqref{stiffmatr3d} and the above that  
\begin{equation}\label{com3dcase}
\begin{split}
(\bs S_\delta)_{ll'}&={\mathcal A}_\delta(\phi_{n_1}\phi_{n_2}\phi_{n_3},\phi_{m_1}\phi_{m_2}\phi_{m_3})\\&=
\frac{h^3}2\int_0^\delta\int_{0}^\pi\int_0^{2\pi}f_{k_1,k_2,k_3}(r\sin\theta\cos\varphi,r\sin\theta\sin\varphi,r\cos\theta)\sin(\theta)\,\d\varphi\d\theta \rho_\delta(r)r^2\,\d r
\\&=\frac{h^3}2\int_0^\delta\int_{0}^\pi\int_0^{2\pi}\hat{f}_{k_1,k_2,k_3}(r,\theta,\varphi)\sin(\theta)\,\d\varphi\d\theta \rho_\delta(r)r^2\,\d r
\\&:=t_{k_1,k_2,k_3},
\end{split}
\end{equation}
where $\hat{f}_{k_1,k_2,k_3}(r,\theta,\varphi)$ is defined in \eqref{hatfin3d}.
 
 It remains to show that the nonlocal stiffness matrix  $\bs S_\delta$ with this ordering leads to the block Toeplitz matrix. 
To this end, we denote the entries of the $N_1N_2N_3$-by-$N_1N_2N_3$ stiffness matrix $\bs S_\delta$ in terms of $n_1,n_2,n_3,m_1,m_2,m_3$ by  $s^{n_1,n_2,n_3}_{m_1,m_2,m_3}.$ 
Then, we have $s^{n_1,n_2,n_3}_{m_1,m_2,m_3}=(\bs S_\delta)_{ll'}=t_{k_1,k_2,k_3}=t_{|n_1-m_1|,|n_2-m_2|,|n_3-m_3|},$ which implies  $\bs S_\delta $ is identical to the block Toeplitz matrix. This completes  the proof.
\end{proof}

Like the 2D case, we study  the sparsity and the limit as $\delta\to 0$.
\medskip

\noindent {\bf (i)\, Sparsity of $\bs S_{\delta}.$}\,    Similar to the 2D case, the entries of the generating matrix $(t_{k_1,k_2,k_3})$ exhibit the following sparsity pattern:
\begin{equation*}\label{gpa3d3d} 
   t_{k_1,k_2,k_3}= \begin{dcases}  \frac{h^3}{2}\int_{\mathbb{B}^3_\delta}f_{k_1,k_2,k_3}(\bss) \rho_\delta(|\bss|)\,\d \bss,\quad & 0\le k_1,k_2,k_3< \delta/h+2,\\
 0,\quad & \max\{k_1,k_2,k_3\}\ge\delta/h+2 .
\end{dcases}\end{equation*}
Moreover, the entries of generating matrix satisfy the following symmetry property  
 \begin{equation}\label{symddinfty}
t_{k_1, k_2,k_3} = t_{k_{\sigma(1)}, k_{\sigma(2)}, k_{\sigma(3)}},\;\;\;\forall 0\leq k_j\leq N_j-1,\;1\leq j\leq 3,
 \end{equation} 
  where $\sigma$ is a permutation representing the symmetry operation.
    Thus, the Toeplitz matrix $\bs S_\delta$ is sparse with a  band-width that increases as the interaction radius  $\delta$ increases.
\medskip

\noindent {\bf (ii)\, Limit of $\bs S_{\delta}$ as 
$\delta\to 0^+$.}\,  
With the above formulas at our disposals, we consider the limiting cases $\delta\to 0$, and show that with the choice of corresponding kernels, we are able to derive the stiffness matrices of the usual  Laplacian operators. For this purpose,  we consider the following kernel functions
\begin{equation}\label{fracker0}
\delta\leq h, \quad   \rho_{\delta}(|\bss|) = C_{3,\delta}^\alpha |\bss|^{-3-\alpha}, \;\;\; \alpha\in [-1,2),
\end{equation}
where $C_{3,\delta}^\alpha=\frac{3(2-\alpha)}{2\pi}\delta^{\alpha-2}$ is chosen so that   $\frac{1}{3}\int_{\mathbb{B}^d_\delta} |\bss|^2 \rho_{\delta}(|\bss|) {\d} \bss=1.$ 
For $0\le \delta\le h,$ similarly, we can use \eqref{bspline4} to simplify $f_{k_1,k_2,k_3}(r,\theta,\varphi)$ in \eqref{cgdefqdd} and evaluate \eqref{gpa3d} directly to derive 
\begin{eqnarray*} 
&& t_{0,0,0}=\frac83+\frac{\alpha-2}2\Big(\frac{-\delta^7}{160\pi(\alpha-9)}+\frac{8\delta^6}{105\pi(\alpha-8)}-\frac{3\delta^5}{32(\alpha-7)}
\\&&\hspace{32pt}+\frac{4(\pi-4)\delta^4}{35\pi(\alpha-6)}+\frac{\delta^3}{\alpha-5}-\frac{8\delta^2}{5(\alpha-4)}-\frac{2\delta}{\alpha-3}\Big),
\\&&t_{0,0,1}=t_{0,1,0}=t_{1,0,0}=(\alpha-2)\Big(\frac{\delta^7}{480\pi(\alpha-9)}-\frac{22\delta^6}{945\pi(\alpha-8)}+\frac{5\delta^5}{192(\alpha-7)}
\\&&\hspace{32pt}-\frac{(9\pi-26)\delta^4}{315\pi(\alpha-6)}-\frac{11\delta^3}{72(\alpha-5)}+\frac{\delta^2}{5(\alpha-4)}+\frac{\delta}{18(\alpha-3)}\Big),
\\&&t_{0,1,1}=t_{1,0,1}=t_{1,1,0}=-\frac16+(\alpha-2)\Big(\frac{-\delta^7}{720\pi(\alpha-9)}+\frac{8\delta^6}{567\pi(\alpha-8)}-\frac{11\delta^5}{768(\alpha-7)}
\\&&\hspace{32pt}+\frac{(27\pi-16)\delta^4}{1890\pi(\alpha-6)}+\frac{\delta^3}{144(\alpha-5)}+\frac{13\delta}{144(\alpha-3)}\Big),
\\&&t_{1,1,1}=-\frac{1}{12}+(\alpha-2)\Big(\frac{-\delta^7}{1080\pi(\alpha-9)}-\frac{8\delta^6}{945\pi(\alpha-8)}+\frac{\delta^5}{128(\alpha-7)}
\\&&\hspace{32pt}+\frac{(9\pi+32)\delta^4}{1260\pi(\alpha-6)}+\frac{\delta^3}{24(\alpha-5)}-\frac{\delta^2}{20(\alpha-4)}+\frac{\delta}{24(\alpha-3)}\Big),
\\&&t_{0,0,2}=t_{0,2,0}=t_{2,0,0}=(\alpha-2)\Big(\frac{-\delta^7}{1920\pi(\alpha-9)}+\frac{4\delta^6}{945\pi(\alpha-8)}-\frac{\delta^5}{384(\alpha-7)}
\\&&\hspace{32pt}-\frac{8\delta^4}{315\pi(\alpha-6)}+\frac{\delta^3}{36(\alpha-5)}-\frac{\delta}{18(\alpha-3)}\Big),
\\&&t_{0,1,2}=t_{0,2,1}=t_{1,0,2}=t_{1,2,0}=t_{2,0,1}=t_{2,1,0}=(\alpha-2)\Big(\frac{\delta^7}{2880\pi(\alpha-9)}-\frac{\delta^6}{405\pi(\alpha-8)}
\\&&\hspace{32pt}+\frac{\delta^5}{768(\alpha-7)}+\frac{\delta^4}{189\pi(\alpha-6)}-\frac{\delta^3}{288(\alpha-5)}-\frac{\delta}{72(\alpha-3)}\Big),
\\&&t_{1,1,2}=t_{1,2,1}=t_{2,1,1}=(\alpha-2)\Big(\frac{-\delta^7}{4320\pi(\alpha-9)}+\frac{4\delta^6}{2835\pi(\alpha-8)}-\frac{\delta^5}{1536(\alpha-7)}
\\&&\hspace{32pt}+\frac{4\delta^4}{945\pi(\alpha-6)}-\frac{\delta^3}{288(\alpha-5)}-\frac{\delta}{288(\alpha-3)}\Big),
\\&&t_{0,2,2}=t_{2,0,2}=t_{2,2,0}=\frac{\alpha-2}2\Big(\frac{-\delta^7}{5760\pi(\alpha-9)}+\frac{\delta^6}{2835\pi(\alpha-8)}-\frac{2\delta^4}{945(\alpha-6)}\Big),
\\&&t_{1,2,2}=t_{2,1,2}=t_{2,2,1}=\frac{\alpha-2}2\Big(\frac{\delta^7}{8640\pi(\alpha-9)}-\frac{\delta^6}{5670\pi(\alpha-8)}-\frac{\delta^4}{1890(\alpha-6)}\Big),
\\&&t_{2,2,2}=\frac{-(\alpha-2)\delta^7}{69120\pi(\alpha-9)}.
\end{eqnarray*}
Thus we have 
\begin{equation}\label{limting0} 
  \lim\limits_{\delta\to 0^+}\bs{S}_{\delta}= \bs S_{x_1}\otimes\bs M_{x_2} \otimes \bs M_{x_3} +\bs M_{x_1}\otimes \bs S_{x_{2}}\otimes\bs M_{x_3}+\bs M_{x_1}\otimes \bs M_{x_{2}}\otimes\bs S_{x_3},
  \end{equation}
  where $\bs S_{x_j}=\tfrac 1 {h}{\rm diag}(-1,2,-1)$ and $\bs M_{x_j}=\tfrac{h}6 {\rm diag}(1,4,1)$ for $j=1,2,3$ are respectively the usual tri-diagonal FEM stiffness and mass matrices in one dimension.

\section{Fast and accurate computation of the stiffness matrix}\label{sec3}
In this section, we describe the implementation details of the efficient algorithm for computing the entries of the generating matrix, denoted as $\{t_{\bk}\}_{\bk \in \Theta_N}$, where 
the index set $\Theta_{N}:=\big\{\bn=(n_1,\cdots, n_d)\,:\, 0\le n_j\le N_j-1,\; 1\le j\le d\big\}.$ 
To fix the idea, we  focus on the hypersingular  kernel $\rho_\delta(|\bss|)$ defined in \eqref{kernel}, as the implementation of regular kernels is  simpler.
  
\subsection{Computation of \texorpdfstring{$t_{k_1,\cdots,k_d}$}. } 
We first introduce the $d$-dimensional spherical coordinates:
\begin{equation}\label{d_sphere}
\begin{split}
&x_{1}=r\cos\theta_{1};\; x_{2}=r\sin\theta_{1}\cos\theta_{2};\;\cdots\cdots; \;x_{d-1}=r\sin\theta_{1}\cdots\sin\theta_{d-2}\cos\theta_{d-1}; 
\\ &x_{d}=r\sin\theta_{1}\cdots\sin\theta_{d-2}\sin\theta_{d-1}, \;\; r=|\bx|, \;\; \theta_1, \cdots, \theta_{d-2}\in [0,\pi],  \;\;  \theta_{d-1}\in[0,2\pi],
\end{split}
\end{equation} 
and  the spherical volume element is 
\begin{equation}\label{coord1}
\begin{split} \d \bx=r^{d-1} \sin ^{d-2}\left(\theta_{1}\right) \sin ^{d-3}\left(\theta_{2}\right) \cdots \sin \left(\theta_{d-2}\right) \d r \,\d \theta_{1}\, \d \theta_{2} \cdots \d \theta_{d-1}:=r^{d-1} \d r\, \d\sigma (\hat\bx). 
\end{split}
\end{equation}
We also define the unit sphere $\mathbb{S}^{d-1}:=\{ \bx \in\RR^d : r=|\bx|=1\}$. 
Using the spherical coordinates, we denote $\hat{f}_{\bk}(r,\hat{\bx})=f_{\bk}(\bx)$.  Then, we present a critical property of the integrand in \eqref{tkjdd}, showing that $\hat{f}_{\bk}(r,\hat{\bx})$ at $r=0$ is a zero with multiplicity two, which plays a decisive role in handling hypersingular integrals, i.e., the kernel function \eqref{kernel}. 
\begin{prop}\label{Smulti}  
For any $ \bk\in \mathbb{N}^d_0$, we have
\begin{equation}\label{binomial1}  
\begin{split}
\hat{f}_{\bk}(r,\hat{\bx})\big|_{r=0}=\partial_r\hat{f}_{\bk}(r,\hat{\bx})\big|_{r=0} =0. 
\end{split}\end{equation} 
Moreover, if $|\bk|_\infty\ge \delta/h+2$, we have 
\begin{equation}\label{fijk3d_dims}
\hat{f}_{\bk}(r,\hat{\bx})=0,\;\;{\rm for\; all}\;(r,\hat{\bx})\in(0,\delta]\times \mathbb{S}^{d-1}.\end{equation} 
\end{prop}
\begin{proof} 
The two-dimensional case corresponding to \eqref{fijk3d_dims} is given in \eqref{sparseprop}, and the tensor-product structure yields the extension to desired result in arbitrary dimensions.
Therefore, we only need to prove \eqref{binomial1}. 
As a direct consequence of \eqref{fijk3d_dims},  we find that for any $|\bk|_\infty\ge 3,$  
\begin{equation*}\label{fk1} 
\begin{split}
\hat{f}_{\bk}(r,\hat{\bx})=0, \;\; \text{for all}\;\;(r,\hat{\bx})\in(0,h]\times\mathbb{S}^{d-1},
\end{split}\end{equation*} 
so we only need to consider the remaining case with $0\leq k_j\leq 2$, $1\leq j\leq d$, that includes the $3^d$ special cases. 
We establish the result by a function-by-function enumeration. The complete set of formulas $\hat{f}_{\bk}(r,\hat{\bx})$ for the two-dimensional case ($d=2$) is provided in the Appendix \ref{appendixA}. For $d\ge 3$, the proof is analogous and relies on explicit analytic expressions for the B-splines. As the details are lengthy and introduce no new ideas, they are omitted here. 
\end{proof}

It is evident that when $\delta \leq h$, the entries of generating matrix $t_{\bs k}$ with hypersingular kernel \eqref{kernel} can be computed exactly. 
Therefore, our focus will be on the case $\delta>h$.
By the virtue of the singularities of the hypersingular kernels at origin, we split the integral \eqref{tkjdd} into
\begin{equation}\label{tkjmulx}\begin{split}
  t_{\bk} =&\frac{h^{d}}2\int_{0}^h\Big(\int_{\mathbb{S}^{d-1}} \hat{f}_{\bk}(r,\hat{\bx})\d\sigma (\hat\bx)\Big)\,\rho_\delta(r)r^{d-1}\,\d r
 \\& +\frac{h^{d}}2\int_{h}^{\delta}\Big(\int_{\mathbb{S}^{d-1}} \hat{f}_{\bk}(r,\hat{\bx})\d\sigma (\hat\bx)\Big)\,\rho_\delta(r)r^{d-1}\,\d r
:= \mathcal{S}_{\bk}+\mathcal{R}_{\bk}.
\end{split}\end{equation} 
Note that when $\delta\leq h$, it reduces to
\begin{equation}\label{tkjmulyy}\begin{split}
  t_\bk=\frac{h^{d}}2\int_{0}^{\delta}\Big(\int_{\mathbb{S}^{d-1}} \hat{f}_{\bk}(r,\hat{\bx})\d\sigma (\hat\bx)\Big)\,\rho_\delta(r)r^{d-1}\,\d r:= \mathcal{S}_\bk.
\end{split}\end{equation}
We compute $\mathcal{S}_{\bk}$ by evaluating the explicit expression of $\hat{f}_{\bk}(r,\hat{\bx})$ on $(r,\hat{\bx})\in (0,\,h\wedge\delta]\times\mathbb{S}^{d-1}$, where $a\wedge b:=\min\{a,b\}$, and we compute $\mathcal{R}_{\bk}$ using a suitable Legendre–Gauss quadrature.  Note that for more complex kernel functions, we can also leverage the analytical expression of 
$\hf_\bk(r,\hx)$ to facilitate the computation of $t_\bk$. 

In Figure~\ref{figdiskball}, the red circle in the left panel and the colored sphere in the right panel mark the integration domains of the integrand $\hat{f}_{\bk}(r,\hat{\bx})$, with the former given by $r\in(0,h)$ and $\theta\in(0,2\pi)$, and the latter by $r\in(0,h)$ and $\hat{\bx}\in(0,2\pi)\times(0,\pi)$. The black mesh indicates the radial scale for the circular or spherical integration region. Moreover, $\hat{f}_{\bk}(r,\hat{\bx})$ is piecewise cubic on this black mesh, and within each grid cell it is cubic B-spline along each coordinate direction.

\begin{figure}[!ht]\hspace{-24pt} 
\begin{minipage}[t]{0.43\textwidth}
\centering
\rotatebox[origin=cc]{-0}{\includegraphics[width=1.2\textwidth,height=0.85\textwidth]{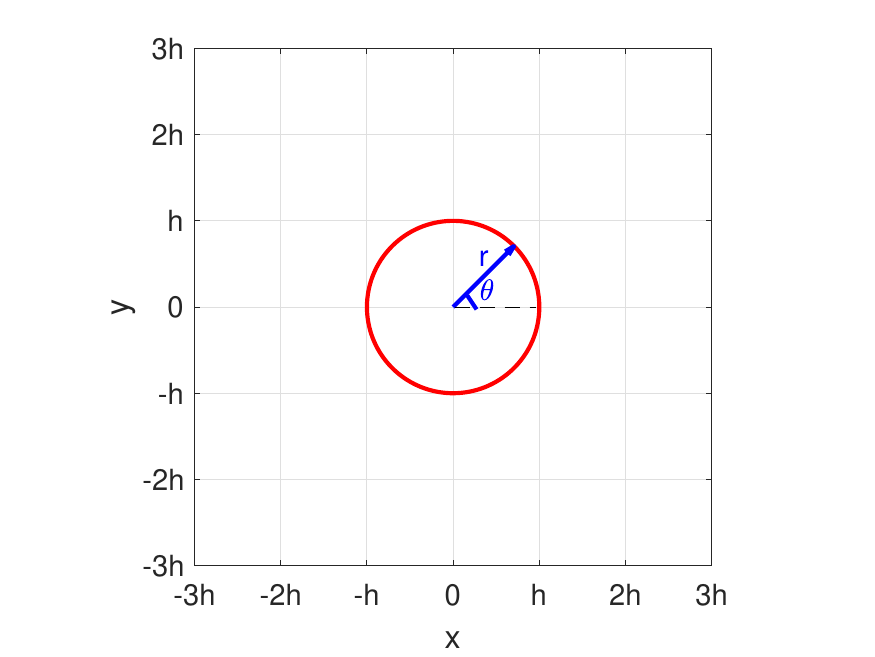}}
\end{minipage}\hspace{0pt}
\begin{minipage}[t]{0.43\textwidth}
\centering
\rotatebox[origin=cc]{-0}{\includegraphics[width=1.2\textwidth,height=0.85\textwidth]{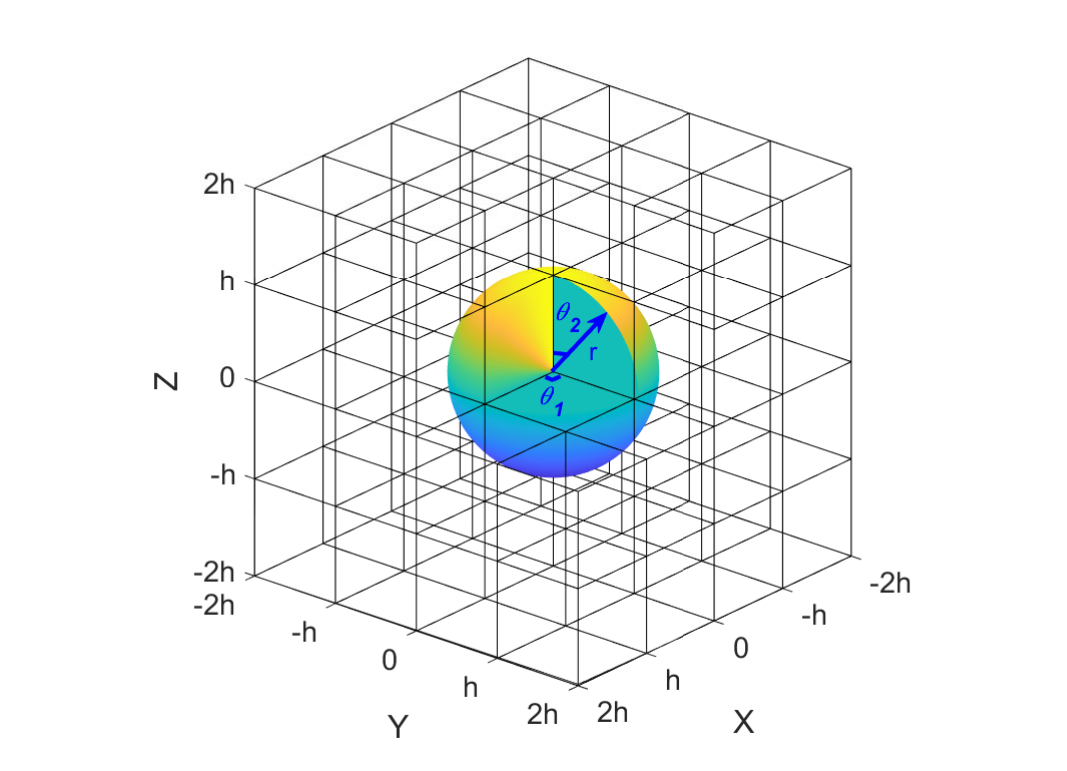}}
\end{minipage}
\vskip -5pt
\caption
{\small The integration domain of $\hat{f}_\bk(r,\hx)$. Left: $\hat{f}_{k_1,k_2}(r,\theta_1)$ with $(r,\theta_1)\in (0,h)\times(0,2\pi)$ ; Right: $\hat{f}_{k_1,k_2,k_3}(r,\theta_1,\theta_2)$ with $(r,\theta_1,\theta_2)\in (0,h)\times(0,\pi)\times(0,2\pi)$.}\label{figdiskball}
\end{figure}

\subsubsection{Computation of the singular part of the integral}
We first consider the computation of the singular integral $\mathcal{S}_{\bk}$. 
Proposition \ref{Smulti} implies that $\hf_{\bk}(r,\hx)=r^2\tilde{f}_{\bk}(r,\hx)$, then this integrand is capable of dealing with the hypersingular kernels \eqref{kernel}. More precisely, we plugging the kernel \eqref{kernel} into \eqref{tkjmulx} to obtain that
\begin{equation}\begin{split}\label{reskj}
\mathcal{S}_\bk=c_{d,\delta}^\alpha h^{d-\alpha}\int_{0}^{h}\Big(\int_{\mathbb{S}^{d-1}}\tilde f_\bk(r,\hat{\bx})\,\d\sigma(\hat{\bx})\Big)r^{1-\alpha}\,\d r,
\end{split}\end{equation}
where we only need to replace the upper limit of the integral in the $r$-direction with $\delta$ when $\delta<h$.
Hence, with the explicit representation of $\tilde f_\bk(r,\hx)$, the integrals $\{\mathcal{S}_\bk\}$ can be evaluated explicitly without any truncation errors.
 Moreover, due to the symmetry properties, it is unnecessary to compute the entire set $\{\mathcal{S}_{\bk}\}_{\bk\in \Theta_2}$ that is $3^d$ elements. 
 To streamline this process, we introduce the following index set that excludes the symmetric elements 
\begin{equation*}
\widetilde{\Theta}_N=\big\{\bn=(n_1,\cdots, n_d)\,:\, 0\leq n_j\leq n_{j+1}\le N,\; 1\leq j\leq d-1 \big\}.
\end{equation*}  
It is clear that $\widetilde{\Theta}_N$ is a subset of $\Theta_N$, and, in particular, we find that $\widetilde{\Theta}_2$ contains only $\frac{(d+2)(d+1)}{2}$ elements.
For instance, in the 2D case, we only need to compute the following $6$ elements:
$$\{\mathcal{S}_{\bk}\}_{\bk\in \widetilde\Theta_2}=\{\mathcal{S}_{0,0},\mathcal{S}_{1,1},\mathcal{S}_{2,2},\mathcal{S}_{0,1},\mathcal{S}_{0,2},\mathcal{S}_{1,2}\},$$
and in the 3D case, we only need to compute the following $10$ elements:
$$\{\mathcal{S}_{\bk}\}_{\bk\in \widetilde\Theta_2}=\{\mathcal{S}_{0,0,0},\mathcal{S}_{1,1,1},\mathcal{S}_{2,2,2},\mathcal{S}_{0,0,1},\mathcal{S}_{0,0,2},\mathcal{S}_{0,1,1},\mathcal{S}_{1,1,2},\mathcal{S}_{0,2,2},\mathcal{S}_{1,2,2},\mathcal{S}_{0,1,2}\}.$$ 
As a result, we highlight that the number of elements needing computation is reduced from $3^d$ to $\mathcal{O}(d^2)$ due to symmetry, a reduction that becomes increasingly advantageous in higher dimensions. 
 
\subsubsection{Computation of the regular part of the integral}
Before we turn to the regular integral $\mathcal{R}_{\bk}$, we first show the regularity of the integrand $\hat{f}_{\bk}(r,\hx)$ with $r>h$. 
Let $\Lambda=I_r\times \mathbb{S}^{d-1}$, with $I_r=(h,\delta)$.  Note that, in this subsection, we assume that $\delta\ge h$, otherwise $\mathcal{R}_{\bk}=0$.
\begin{prop}\label{Singlaadd}  
For any $\bk\in \mathbb{N}^d_0$, the integrand $\hf_{\bk}(r,\hx)\in C^2(\Lambda)$, with $\Lambda=(h,\delta)\times\mathbb{S}^{d-1}$. Thus $\hf_{\bk}(r,\hx)\in H^3(\Lambda)$,
where $H^3(\Lambda)$ is the usual Sobolev space.
\end{prop}
\begin{proof}
In view of \eqref{bspline4},  it is easy to check that $B_3(t)\in C^2(\mathbb{R})$. Then, by the chain rule, one verifies readily that $\hf_\bk(r,\hx)\in C^2(\Lambda)$. The second regularity result is a direct consequence of the first one. This ends the proof.
\end{proof}

With the above regularity, it is sufficient to use Legendre-Gauss quadrature to approximate $\mathcal{R}_{\bk}$. More precisely, 
\begin{equation}\label{rijk3d}\begin{split}
  \mathcal{R}_{\bk}\approx  \mathcal{R}^\ast_{\bk}:= \frac{h^{d}}2 \sum_{p=0}^{N_r}\sum_{\ell_1=0}^{N_{\theta_1}}\cdots\sum_{\ell_{d-1}=0}^{N_{\theta_{d-1}}}\check{f}_\bk(r_p,\theta_{\ell_1},\cdots,\theta_{\ell_{d-1}})\rho_{\delta}(r_p)r_p^{d-1}\,\omega_p^r\omega_{\ell_1}\cdots\omega_{\ell_{d-1}},  
\end{split}\end{equation}
where $$\check{f}_\bk(r_p,\theta_{\ell_1},\cdots,\theta_{\ell_{d-1}})=\hat{f}_\bk(r_p,\theta_{\ell_1},\cdots,\theta_{\ell_{d-1}})\sin ^{d-2}(\theta_{\ell_1}) \sin ^{d-3}(\theta_{\ell_2}) \cdots \sin(\theta_{\ell_{d-2}}).$$
In the above, $\{r_p,\omega^r_p\}_{p=0}^{N_r}$, $\{\theta_{\ell_j},\omega_{\ell_j}\}_{\ell_j=0}^{N_j}$, $0\leq j\leq d-1$ denote the Legendre-Gauss quadrature nodes and weights in $(h,\delta)$ and $\mathbb{S}^{d-1}$, respectively.

Let $I=(a,b).$
For any $u\in C(I),$ we define the interpolation operator $\mathcal{I}_N:\,C(I)\rightarrow \mathcal{P}_N$ such that 
\begin{equation}
\label{interpo}
\mathcal{I}_{N}u(x_j)=u(x_j),\quad 0\leq j\leq N,
\end{equation}
where $\{x_j\}$ denote the Legendre-Gauss quadrature points over $I,$ and $\mathcal{P}_N$ denote the set of all algebraic polynomials of degree $\leq N$.
According to \cite[Theorem 3.41]{ShenTangWang2011},  if $\partial_x^mu\in L^2_{\omega^m}(I)$ with $m\ge1$ and $\omega^m(x)=(b-x)^{\frac{m}2}(x-a)^{\frac{m}2}$, we have
\begin{equation}\label{interr}
\|\mathcal{I}_Nu-u\|_{L^2(I)}\leq cN^{-m}\|\partial_x^mu\|_{L^2_{\omega^m}(I)},
\end{equation}
where $c$ is a positive constant independent of $m$, $N$ and $u$. Note that, in \eqref{interr}, we converted the standard interval $(-1,1)$ to $I=(a,b)$ by using the linear transformation.

We define the $d$-dimensional interpolation $\mathcal{I}^d_{N}:=\mathcal{I}_{N_r}\!\circ \mathcal{I}_{N_{\theta_1}}\circ\cdots\circ\mathcal{I}_{N_{\theta_{d-1}}}:C(\Lambda)\rightarrow \mathcal{P}_{N_r}\times\mathcal{P}_{N_{\theta_1}}\times\cdots\times\mathcal{P}_{N_{\theta_{d-1}}}$, which satisfies 
\begin{equation}
\label{interpoddim}
\mathcal{I}^d_{N}u(r_p,\theta_{\ell_1},\cdots,\theta_{\ell_{d-1}})=u(r_p,\theta_{\ell_1},\cdots,\theta_{\ell_{d-1}}),\quad 0\leq p\leq N_r,\;0\leq \ell_j\leq N_{\theta_j},\;1\leq j\leq d-1,
\end{equation}
where $\mathcal{I}_{N_r}$ and $\mathcal{I}_{N_{\theta_j}}$ denote the one dimensional interpolations in $r$- and spherical-direction as in \eqref{interpo}, respectively.

To describe the error more precisely, we introduce the space $B^m(\Lambda)$ for $m\ge d$ with the semi-norm and norm
\begin{equation*}\begin{split}
|u|_{B^{m}(\Lambda)} = \Big( \sum_{j=1}^d \sum_{\bmm \in \mathcal{E}_j} \|\partial^\bmm_\bx u\|^2_{\omega^{m_j{\bs e}_j}}\Big)^{1/2},\;\;\;\|u\|_{B^{m}(\Lambda)} = 
\Big( \|u\|^2_{L^2(\Lambda)} + |u|^2_{B^m(\Lambda)}\Big)^{1/2},
\end{split}\end{equation*}
where 
$$\mathcal{E}_j = \Big\{ \bmm \in \mathbb{N}_0^d : d \leq m_j \leq m; \;m_i \in \{0, 1\}, i \neq j; \;\sum_{k=1}^d m_k = m\Big\}.$$
In the above, $\bx=(r,\theta_1,\cdots,\theta_{d-1})$ and the weight functions are given by
\begin{equation*}\begin{split}
&\omega^\bmm(\bx):=\omega_r^{m_1}(r)\omega^{m_2}_{\theta_1}(\theta_1)\cdots\omega^{m_d}_{\theta_{d-1}}(\theta_{d-1}),\;\;\;\omega^{m_1}_r(r)=(\delta/h-r)^{m_1}(r-1)^{m_1},
\\&\omega^{m_d}_{\theta_{d-1}}(\theta_{d-1})=(2\pi-\theta_{d-1})^{m_d}\theta_{d-1}^{m_d},\;\;\;\omega^{m_{j+1}}_{\theta_j}(\theta_{j})=(\pi-\theta_j)^{m_{j+1}}\theta_j^{m_{j+1}}, \;\;\;1\leq j\leq d-2.
\end{split}\end{equation*}
\begin{thm}\label{rerr} If $\rho_\delta(hr)\in C^2(I_r)$. For any $\bk\in\mathbb{N}^d_0$, we have
\begin{equation}
\begin{split}
|\mathcal{R}_\bk-{\mathcal{R}}^\ast_\bk|\leq cN^{-3}|F_\bk|_{B^3(\Lambda)},
\end{split}
\end{equation}
where $c$ is a positive constant independent of $N$ and $$F_\bk(r,\hx):=F_\bk(r,\hx;h)=\rho_{\delta}(r)r^{d-1} \sin ^{d-2}\left(\theta_{1}\right) \sin ^{d-3}\left(\theta_{2}\right) \cdots \sin \left(\theta_{d-2}\right)\hf_\bk(r,\hx).$$ 
\end{thm}
\begin{proof} 
By using an argument similar to the proof of Theorem 8.6 in \cite{ShenTangWang2011}, we obtain from \eqref{interr} that
\begin{equation}
\begin{split}
&\|\mathcal{I}^d_{N}u-u\|\leq c N^{-m}|u|_{B^m(\Lambda)}. 
\end{split}
\end{equation}
With the aid of Proposition \ref{Singlaadd}, one verifies readily that $F_\bk(r,\hx)\in B^3(\Lambda)$,
which together with the above estimates and Cauchy-Schwarz inequality leads to
\begin{equation*}
\begin{split}
|\mathcal{R}_\bk-{\mathcal{R}}^\ast_\bk|&=\Big|\frac{h^{d}}{2}\int_{h}^{\delta}\int_{\mathbb{S}^{d-1}}(\mathcal{I}-\mathcal{I}^d_{N})F_\bk(r,\hx)\d\hx \,\d r\Big|
\\&\leq c\Big(\int_{h}^{\delta}\int_{\mathbb{S}^{d-1}}\big((\mathcal{I}-\mathcal{I}^d_{N})F_\bk(r,\hx)\big)^2\d\hx \,\d r\Big)^{\frac12}\leq cN^{-3}|F_\bk|_{B^3(\Lambda)}.
\end{split}
\end{equation*}
 This ends the proof.
\end{proof}

Consequently, we summarise the algorithm for computing the entries  $\{t_\bk\}_{\bk\in\widetilde{\Theta}_N}$ of the generating matrix $\bs G$ as follows.

\LinesNotNumbered
\begin{algorithm}[H]
\caption{Evaluate $\{t_\bk\}$ with $\bk\in \widetilde{\Theta}_N$}\label{Alg}
\small
\setlength{\algomargin}{0.1em}  
\KwIn{the value of $\delta>0$ and mesh size $h$}

Compute $\{\mathcal{S}_\bk\}_{\bk\in\widetilde{\Theta}_2}$ by \refe{reskj}, and set $\mathcal{S}_\bk=0$ with $|\bk|_{\infty}\ge3$ \tcc*{Singular part}
\If{$\delta\leq h$}{
  Set $t_\bk=\mathcal{S}_\bk$ with $\bk\in\widetilde{\Theta}_2$       \tcc*{Update $\{t_\bk\}$}
}
\Else{
\For{$k_1=0,\cdots,N$}{
\For{$k_2=0,\cdots,k_1$}{
 \hspace{12pt}$\vdots$
 
\For{$k_d=0,\cdots,k_{d-1}$}{
Compute ${\mathcal{R}}^\ast_\bk$ by \refe{rijk3d} \tcc*{Regular part}

Set  $t_\bk=\mathcal{S}_\bk+{\mathcal{R}}^\ast_\bk$ \tcc*{Update $\{t_\bk\}$}

}
      \hspace{12pt}$\vdots$  
     }
}

}
\KwOut{$\{t_\bk\}$ with $\bk\in \widetilde{\Theta}_N$}
\end{algorithm}
\medskip

\begin{figure}[!ht] \vspace{-10pt}
\subfigure[fixed $\mathcal{R}_{0,0}^\ast$]{
\begin{minipage}[t]{0.43\textwidth}
\centering
\rotatebox[origin=cc]{-0}{\includegraphics[width=1\textwidth,height=0.7\textwidth]{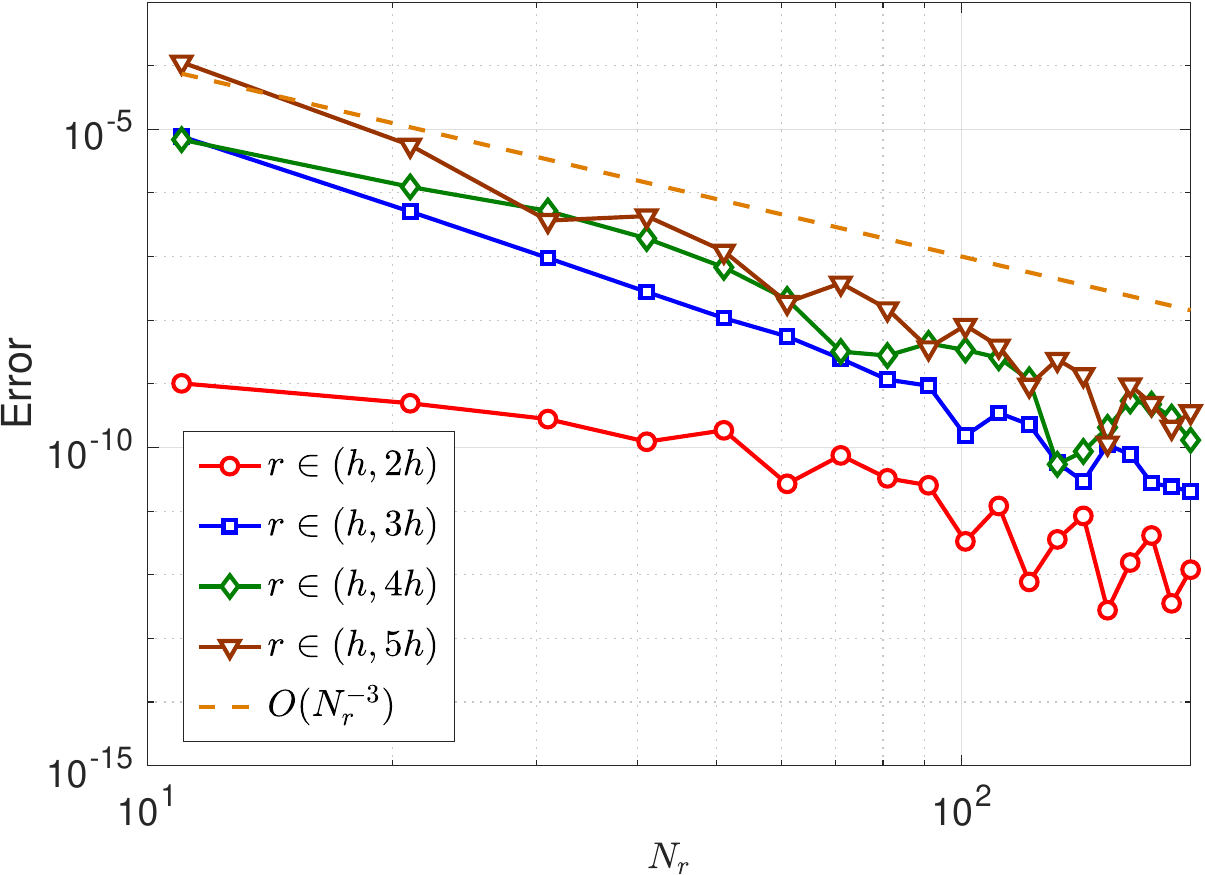}}
\end{minipage}} 
\subfigure[fixed $r\in(h,2h)$]{
\begin{minipage}[t]{0.43\textwidth}
\centering
\rotatebox[origin=cc]{-0}{\includegraphics[width=1\textwidth,height=0.7\textwidth]{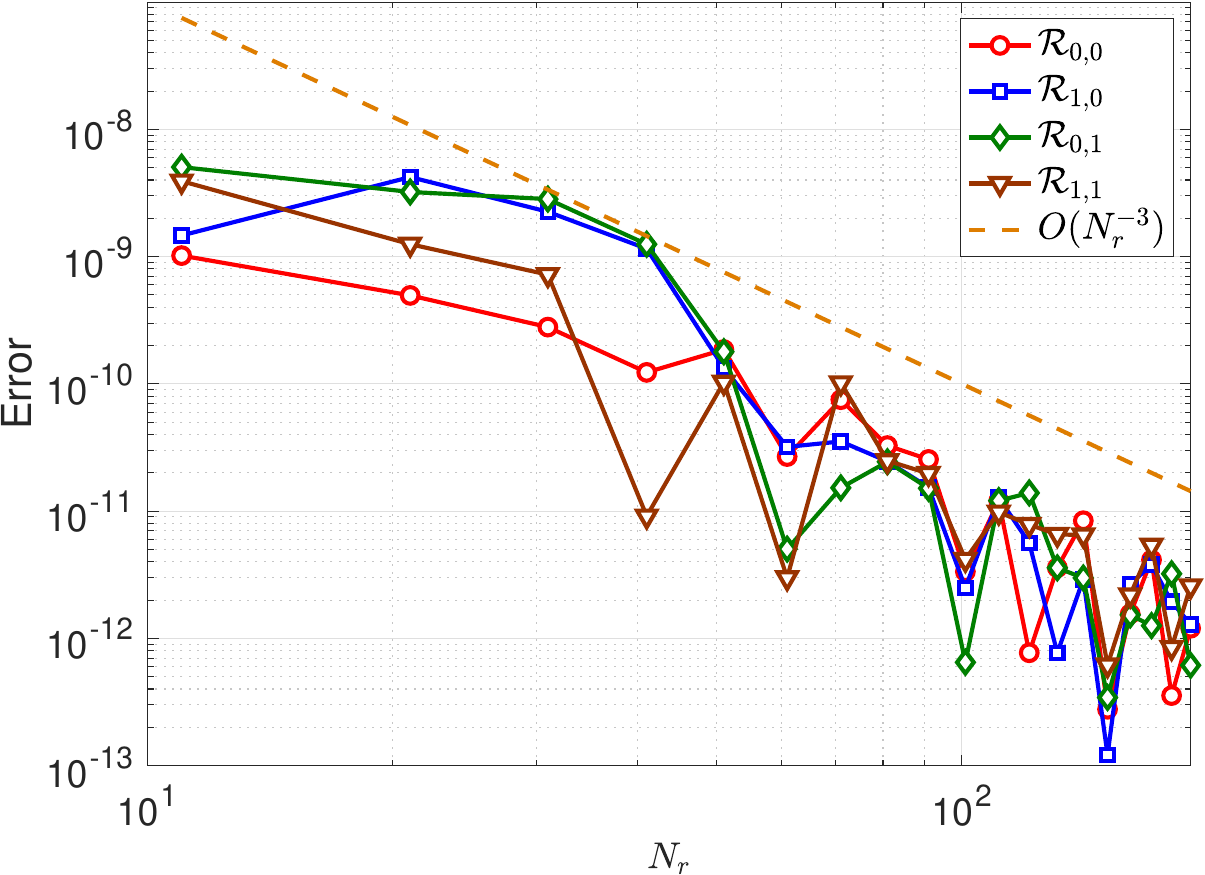}}
\end{minipage}} 
\subfigure[fixed $\mathcal{R}_{0,0}^\ast$]{
\begin{minipage}[t]{0.43\textwidth}
\centering
\rotatebox[origin=cc]{-0}{\includegraphics[width=1\textwidth,height=0.7\textwidth]{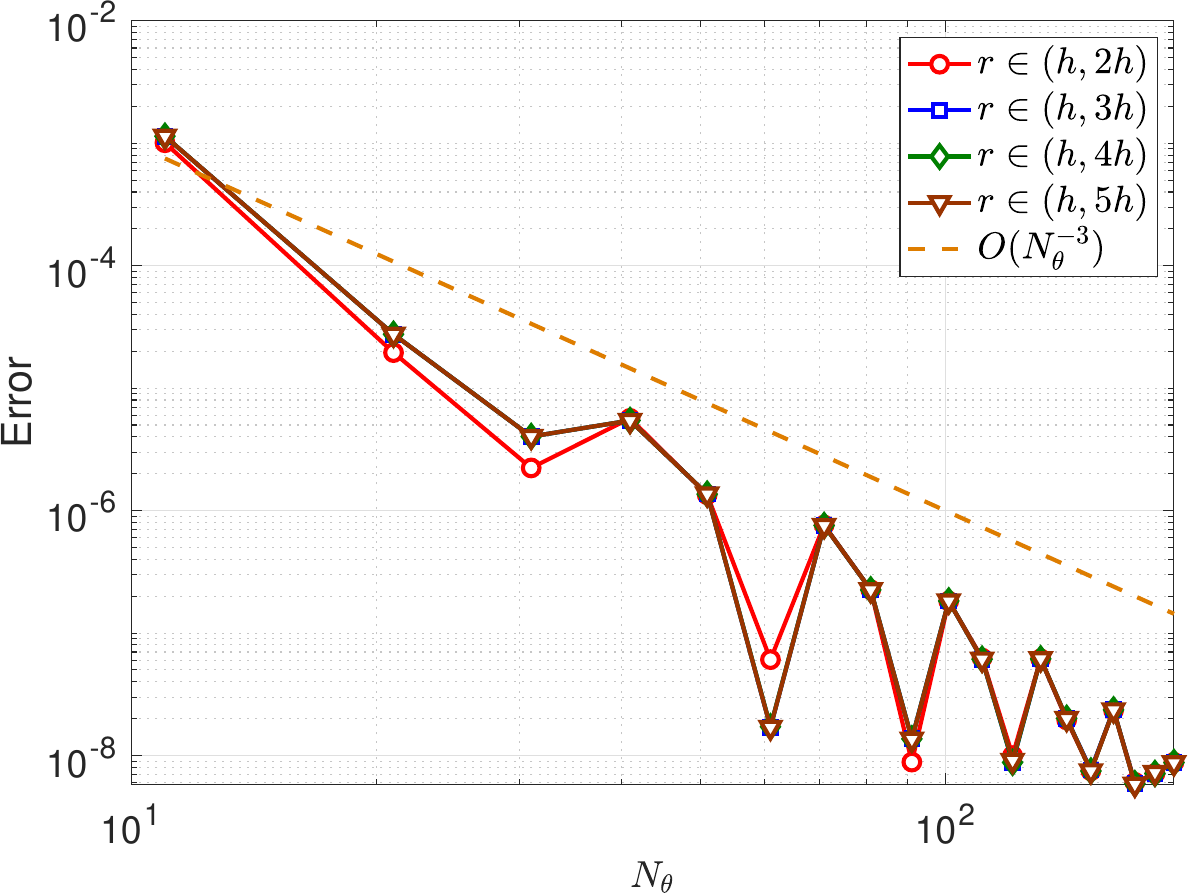}}
\end{minipage}} \vspace{-10pt}
\subfigure[fixed $r\in(h,2h)$]{
\begin{minipage}[t]{0.43\textwidth}
\centering
\rotatebox[origin=cc]{-0}{\includegraphics[width=1\textwidth,height=0.7\textwidth]{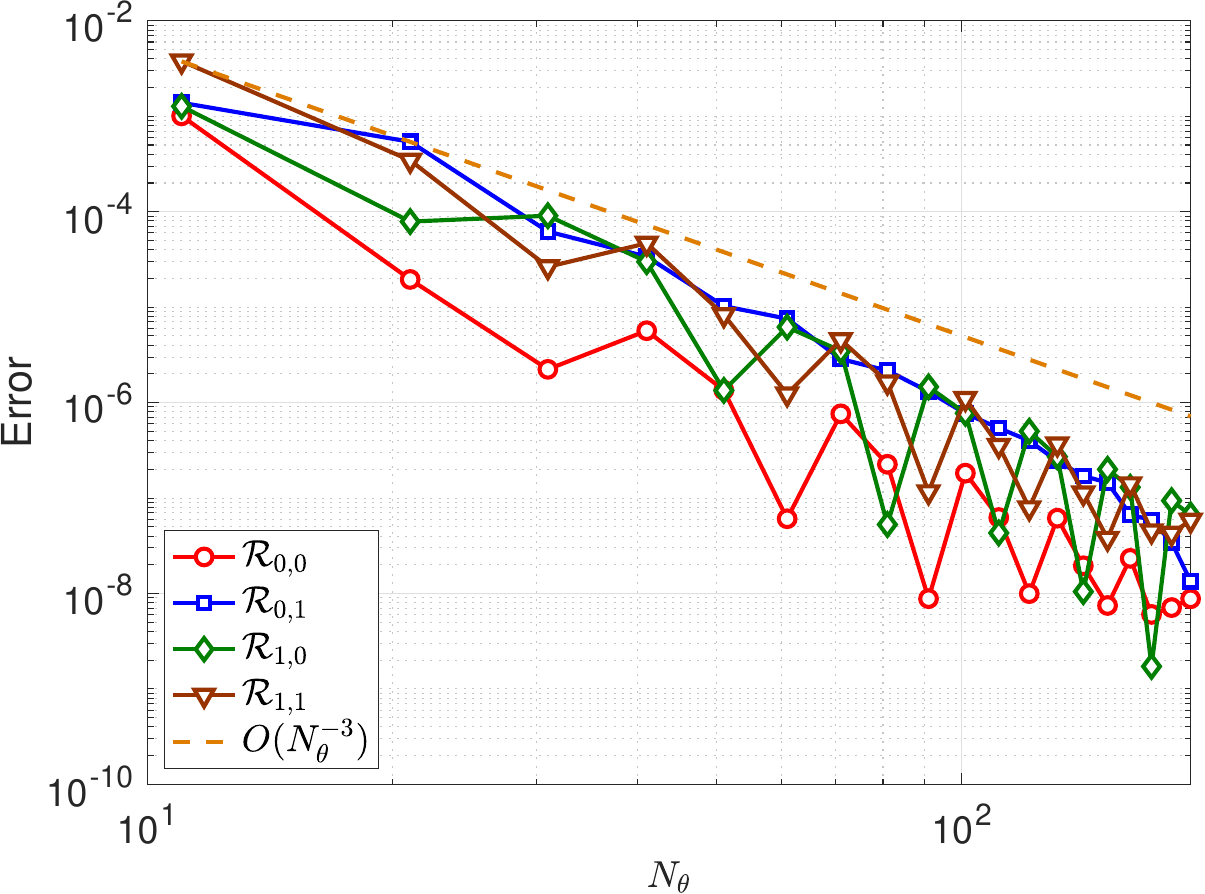}}
\end{minipage}}
\caption
{\small Errors of Legendre-Gauss quadrature for $\mathcal{R}^\ast_{k_1,k_2}$ in two dimensions. (a). against $N_r$ with different $r$; (b). against $N_r$  with different $k_1,k_2$; 
(c). against $N_\theta$  with different $r$; (d). against $N_\theta$ with different $k_1,k_2$. }\label{figquaderr2d}
\end{figure}

We conclude this section with an examination of the Legendre-Gauss quadrature applied to the regular component of the integrand, $\hat{f}_{\bk}(r,\hat{\bs{x}})$, namely $\mathcal{R}^\ast_{\bk}$, in both two- and three-dimensional settings.
We take $\rho_\delta(r)=r^{-d-0.7}$ and $h=1$. Clearly, the given kernel $\rho_\delta(r)$ is sufficiently smooth in $I_r$. Note that the exact value of $\mathcal{R}_{\bk}$ is unknown; therefore, we adopt a numerical approximation with $N_r = N_\theta = 500$ in 2D and $N_r = N_{\theta_1} = N_{\theta_2} = 500$ in 3D as the reference solution.
In Figure \ref{figquaderr2d}, we plot the errors of Legendre-Gauss quadrature against various $N_r$ and $N_\theta$ in 2D and 3D, respectively.  Here, we take $I_r=(1,2),$ $(1,3),$ $(1,4)$ and $(1,5).$
Figure \ref{figquaderr2d} indicate the algebraic order of convergence for different $k_1,k_2,k_3$, which seem better than $O(N^{-3}_r)$ and $O(N^{-3}_\theta)$. 
Indeed, we only need a small number of Legendre-Gauss nodes to get high accuracy as we can take full advantage of the small multiplier $h^3$.

\begin{figure}[!ht] 
\subfigure[fixed $\mathcal{R}_{1,1,1}^\ast$]{
\begin{minipage}[t]{0.43\textwidth}
\centering
\rotatebox[origin=cc]{-0}{\includegraphics[width=1\textwidth,height=0.7\textwidth]{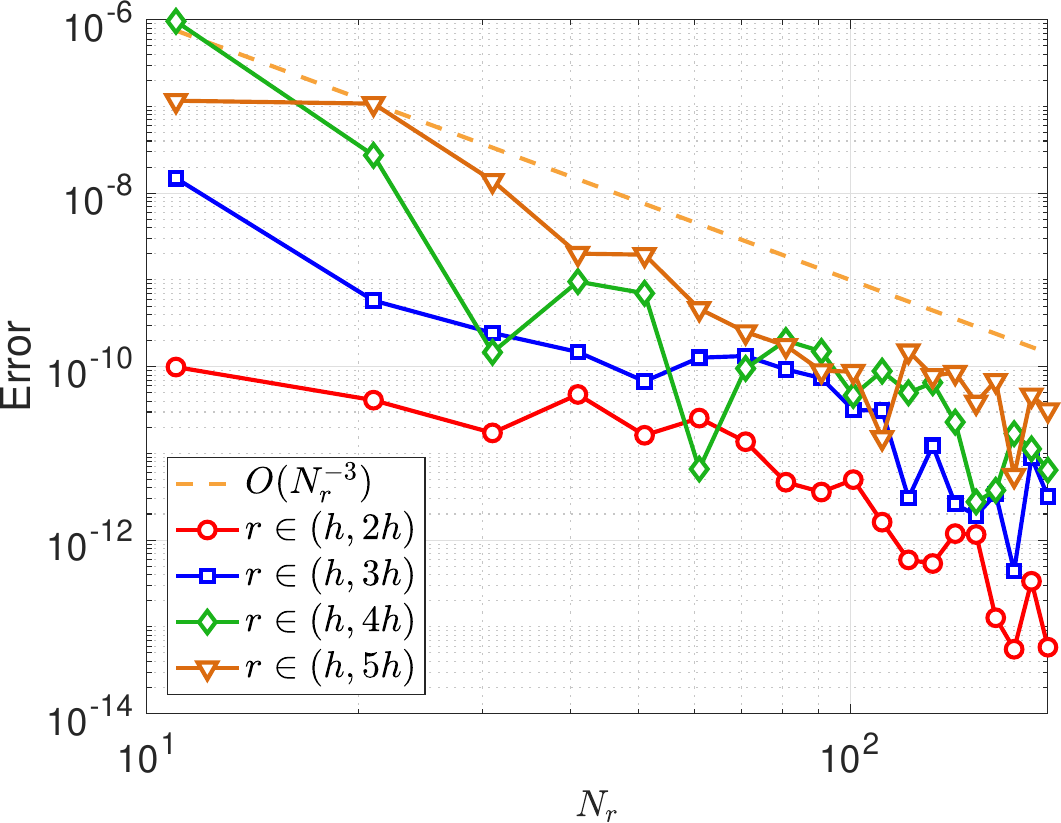}}
\end{minipage}}
\subfigure[fixed $r\in(h,5h)$]{
\begin{minipage}[t]{0.43\textwidth}
\centering
\rotatebox[origin=cc]{-0}{\includegraphics[width=1\textwidth,height=0.7\textwidth]{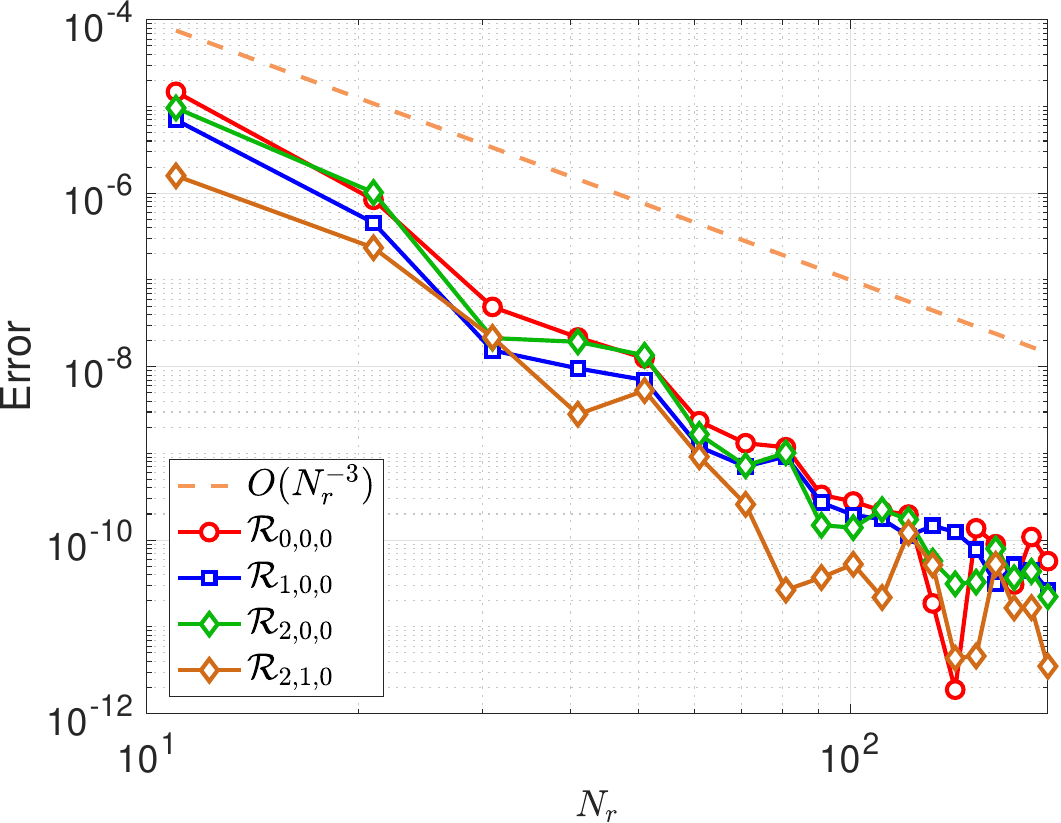}}
\end{minipage}}\hspace{-28pt}
\subfigure[fixed $\mathcal{R}_{1,1,1}^\ast$]{
\begin{minipage}[t]{0.43\textwidth}
\centering
\rotatebox[origin=cc]{-0}{\includegraphics[width=1\textwidth,height=0.7\textwidth]{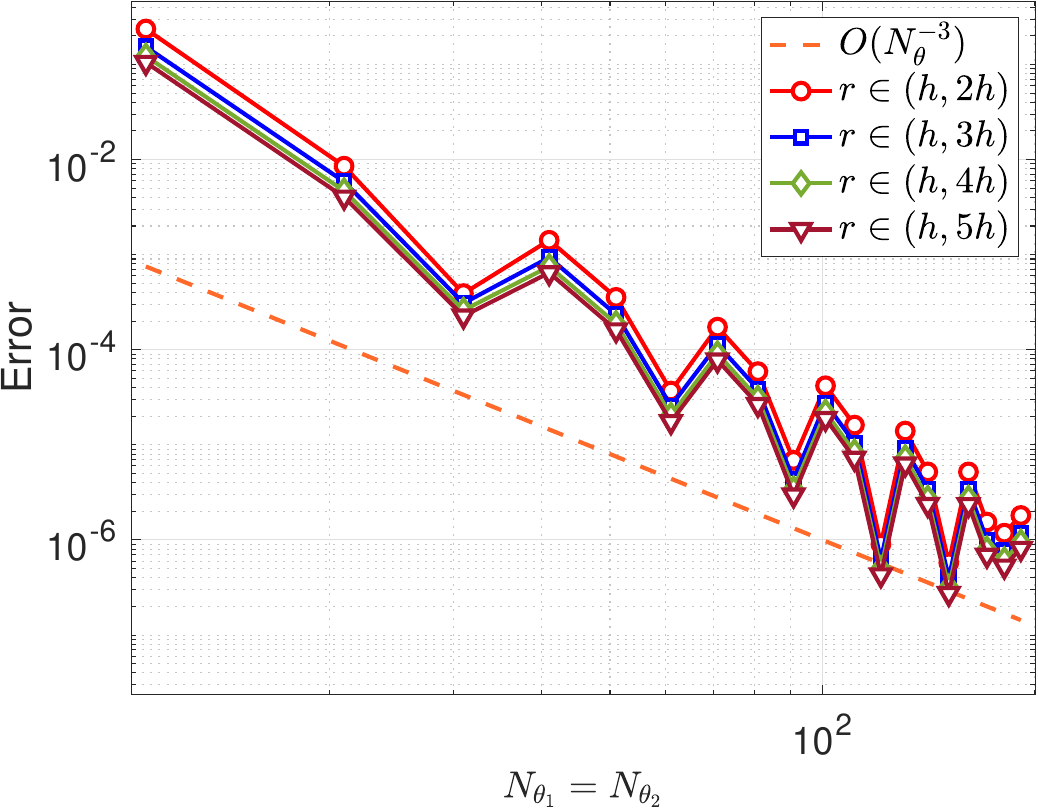}}
\end{minipage}}
\subfigure[fixed $r\in(h,5h)$]{
\begin{minipage}[t]{0.43\textwidth}
\centering
\rotatebox[origin=cc]{-0}{\includegraphics[width=1\textwidth,height=0.7\textwidth]{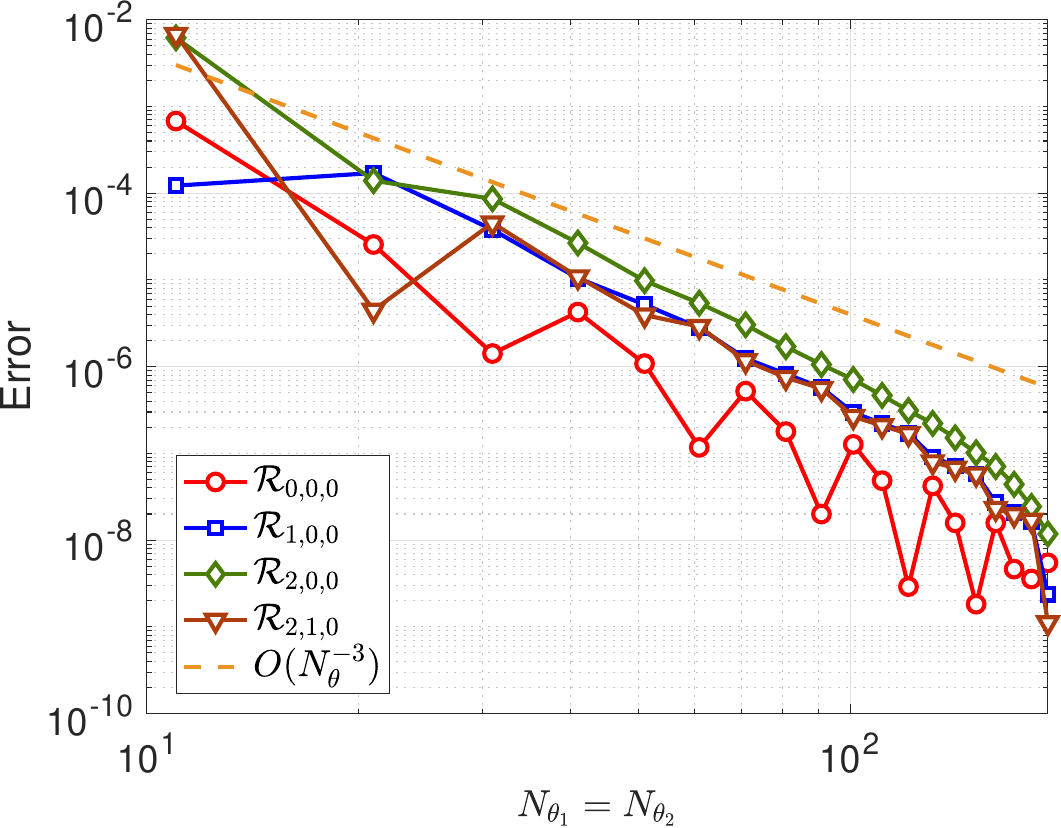}}
\end{minipage}}
\caption
{\small Errors of Legendre-Gauss quadrature for $\mathcal{R}^\ast_{k_1,k_2,k_3}$ in three dimensions. (a). against $N_r$ with different $r$; (b). against $N_r$  with different $k_1,k_2,k_3$; 
(c). against $N_\theta$  with different $r$; (d). against $N_\theta$ with different $k_1,k_2,k_3$. }\label{figquaderr3d}
\end{figure}

\medskip

\section{Numerical results and concluding remarks}\label{sec5}
In this section, we consider several examples of PDEs with nonlocal operator and show the accuracy and efficiency of the algorithm for computing the FEM stiffness matrix. We conduct convergence tests employing both two- and three-dimensional contrived solutions, with both given exact solution and right hand source function.

\subsection{Accuracy test  in 2D and 3D} 
 We now present some numerical results and begin by testing the accuracy of the algorithm for equation \eqref{uvshx}. To this end, we consider an exact solution given by  
$u(\bx) = \mathrm{e}^{-\lambda^2|\bs x|^2},$  
and choose the kernel  $\rho_\delta(\bx) = \frac{3d}{(d-1)\pi\delta^3|\bx|^{d-1}}, \;\; d = 2, 3,$  
which satisfies condition \eqref{ker2}.  When $d = 2$, then $\rho_\delta(\bx) = \frac{6}{\pi\delta^3|\bx|}$ substituting this kernel function into the nonlocal operator defined in \eqref{NonLOper} yields
\begin{equation}\label{benchmark2d}\begin{split}
&\mathcal{L}_\delta\big\{ \e^{-\lambda^2|\bs x|^2} \big\}=\frac{3}{\pi\delta^3}\int_{\mathbb{B}^2_\delta} \big(\e^{-\lambda^2|\bs x+\bss|^2}-\e^{-\lambda^2|\bs x|^2}\big)|\bss|^{-1}\, \d \bss
\\&=\frac{3}{\pi\delta^3}\int^{2\pi}_0\int_0^\delta \big(\e^{-\lambda^2(x_1+r\cos\theta)^2-\lambda^2(x_2+r\sin\theta)^2}-\e^{-\lambda^2(x_1^2+x_2^2)}\big)\, \d r\,\d \theta
\\&=\frac{3\e^{-\lambda^2(x_1^2+x_2^2)}}{2\pi\delta^3}\int^{2\pi}_0 \big(2\lambda\delta+E(x_1,x_2,\theta)\e^{-\lambda^2(x_1\cos\theta +x_2\sin\theta )^2}\big)\,\d \theta
=f_\lambda (\bs x),
\end{split}\end{equation}
where $E(x_1,x_2,\theta)=\sqrt{\pi}{\rm erf}(\lambda(x_1\cos\theta +x_2\sin\theta))-\sqrt{\pi}{\rm erf}(\lambda(\delta+x_1\cos\theta +x_2\sin\theta)),$ with ${\rm erf}(\cdot)$ denote the Gauss error function.  
\begin{figure}[!ht]\hspace{-24pt}
\begin{minipage}[t]{0.43\textwidth}
\centering
\rotatebox[origin=cc]{-0}{\includegraphics[width=1.0\textwidth,height=0.75\textwidth]{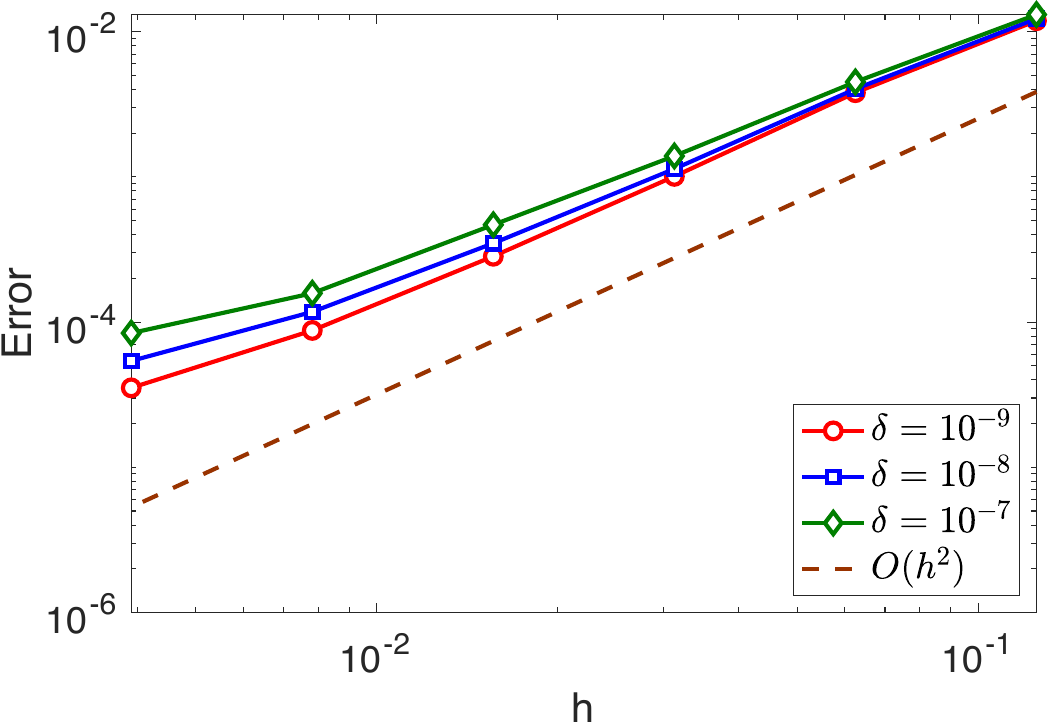}}
\end{minipage}\hspace{12pt}
\begin{minipage}[t]{0.43\textwidth}
\centering
\rotatebox[origin=cc]{-0}{\includegraphics[width=1.0\textwidth,height=0.75\textwidth]{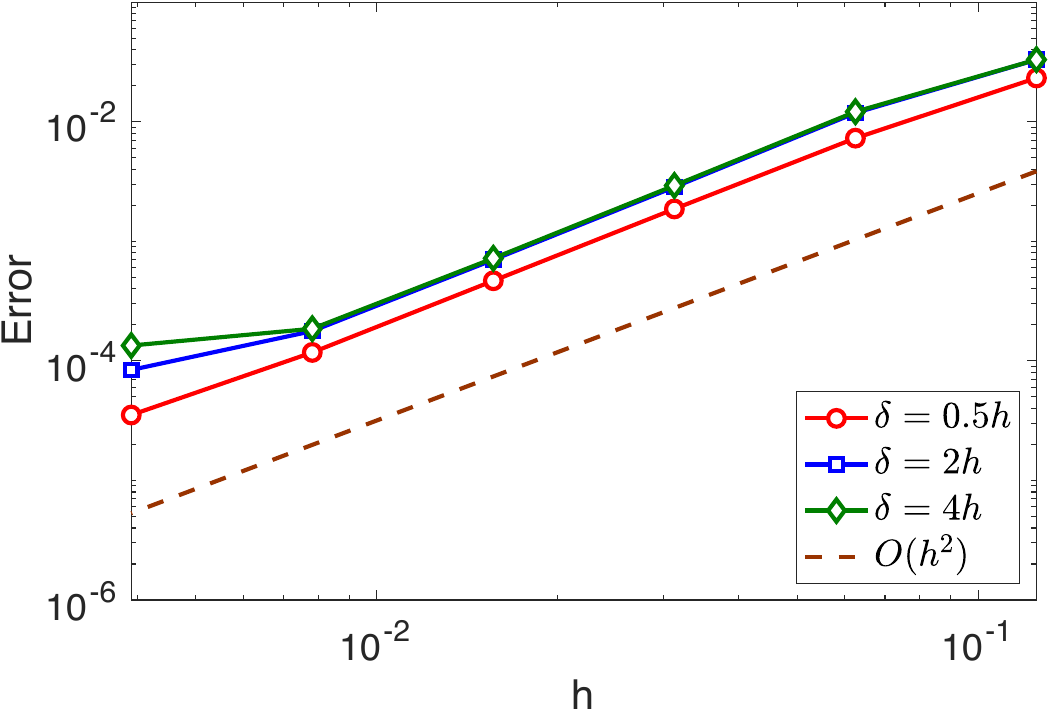}}
\end{minipage}
\vskip -5pt
\caption
{\small  Errors and convergence order of FEM approximation to \eqref{1DNon} with an ``exact" solution given in \eqref{benchmark2d}.
  Left:  Fixed $\delta$. Right:  Fixed $\delta=\nu h$ with various $\nu$.} \label{figerr2d}
\end{figure} 

In the above, the one-dimensional integrals for $f_\lambda(\bs x)$ can be evaluated with high accuracy using Legendre-Gauss quadrature, even with a small number of quadrature nodes. This allows us to select a relatively large parameter $\lambda > 0$, such that the function $u_\lambda(\bs x) := \mathrm{e}^{-\lambda^2|\bs x|^2} \approx 0$ for $\bs x \in \Omega_\delta$, while both $u_\lambda$ and $f_\lambda$ remain sufficiently smooth in $\Omega \cup \Omega_\delta$. Consequently, we expect the finite element approximation to achieve optimal second-order convergence for any $\delta$, provided that the stiffness matrix is computed with sufficient accuracy. In this test, we set $\lambda = 12$. As shown in Figure~\ref{figerr2d}, the results confirm that optimal second-order convergence is attained for $d=2$, both for various fixed values of $\delta$ and for different fixed radius $\nu = \delta / h$.

When $d=3$, constructing an analytical solution as in \eqref{benchmark2d} becomes challenging, primarily because it involves a double integral; when combined with the variational formulation, it ultimately requires evaluating a five-dimensional integral, which significantly increases computational complexity. Therefore, we choose a simple right-hand side $f(\bx) = 1$ and the kernel function $\rho_\delta(\bx)=\frac{9}{4\pi\delta^3|\bx|^2}$. Since the exact solution is unknown even for a constant kernel function, we adopt the numerical solution computed on a sufficiently fine mesh as the reference solution.  In Figure\,\ref{figerr3d}, we plot the discrete $L^2$-errors on a log-log scale against $h$ for various small values of $\delta$. The results clearly demonstrate second-order convergence for both regimes: $\delta <h$ and $\delta\ge h$.

\begin{figure}[!ht]\hspace{-24pt}
\begin{minipage}[t]{0.43\textwidth}
\centering
\rotatebox[origin=cc]{-0}{\includegraphics[width=1.0\textwidth,height=0.75\textwidth]{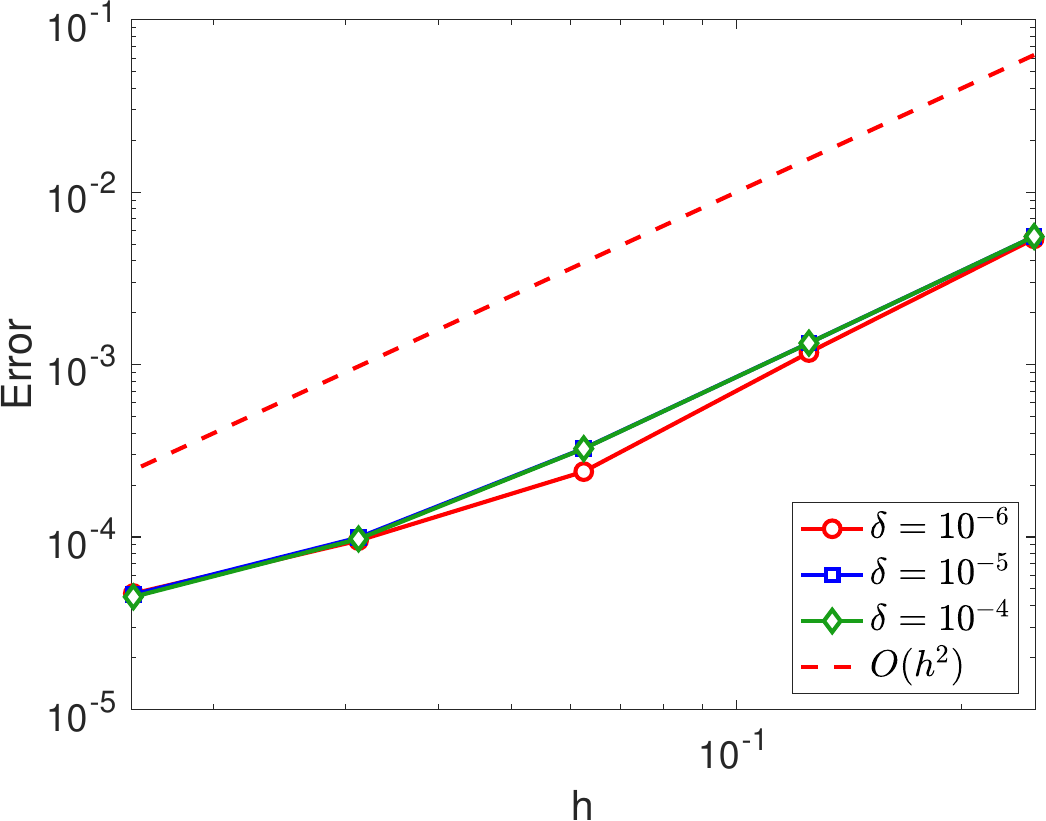}}
\end{minipage}\hspace{12pt}
\begin{minipage}[t]{0.43\textwidth}
\centering
\rotatebox[origin=cc]{-0}{\includegraphics[width=1.0\textwidth,height=0.75\textwidth]{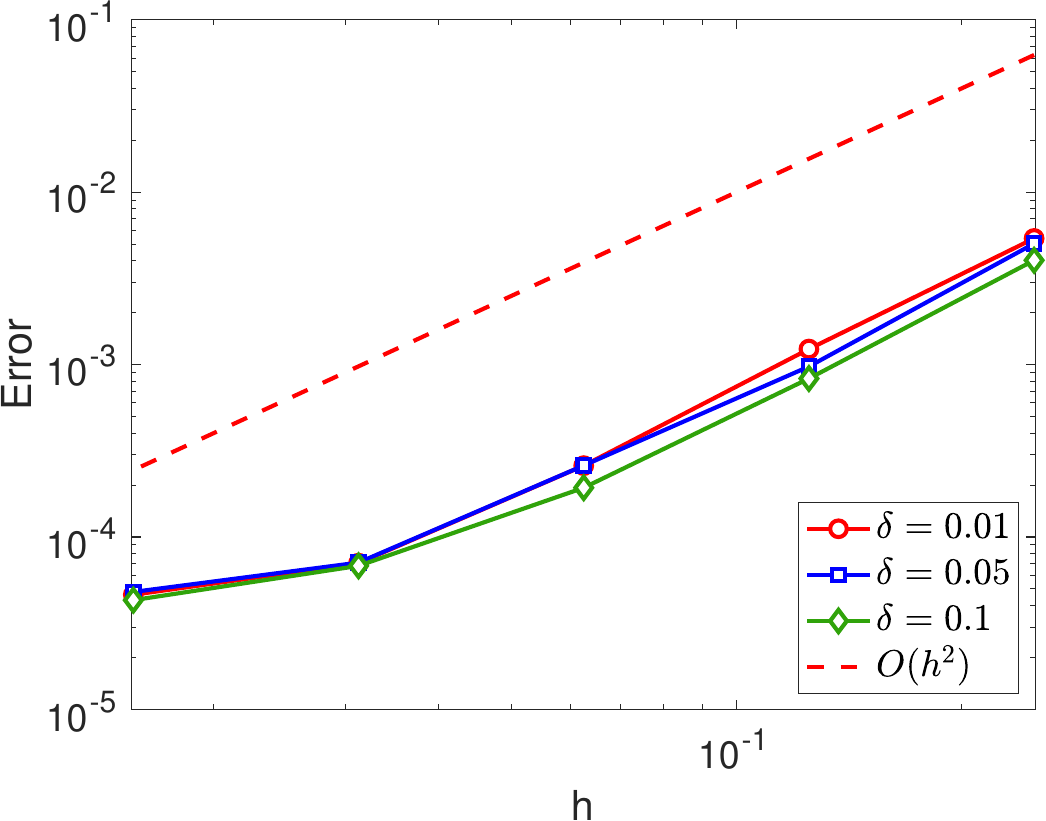}}
\end{minipage}
\vskip -5pt
\caption
{\small  Errors and convergence order of FEM approximation to \eqref{1DNon} with given source function $f(\bx)=1$.
  Left:  Fixed small $\delta$. Right:  Fixed big $\delta$.} \label{figerr3d}
\end{figure}

\subsection{Numericial results for hypersingular kernel}
Next, we consider the source problem with $f(\bx) = 1$ and a hypersingular kernel given by $\rho_\delta(\bss) = c^\alpha_\delta |\bss|^{-d -\alpha}$, where $s \in (0,1)$. In this setting, the exact solution is not available in closed form. It is known that the nonlocal operator $\mathcal{L}_\delta$ converges to the fractional Laplacian $(-\Delta)^{\frac{\alpha}{2}}$ as $\delta \to \infty$. As established in \cite[(7.12)]{grubb2015fractional}, the solution to the fractional Laplacian equation $(-\Delta)^{\frac{\alpha}{2}}u=f$ exhibits a singular behavior near the boundary $\partial \Omega$, and admits the form $u(\bx) = \mathrm{dist}(\bx, \partial \Omega)^{\frac{\alpha}{2}} + v(\bx),$ where $\mathrm{dist}(\bx, \partial \Omega)$ denotes the distance from $\bx \in \Omega$ to the boundary $\partial \Omega$, and $v$ is a smooth function. On the other hand, as $\delta \rightarrow 0$, the nonlocal operator recovers the classical local operator uniformly for all $\alpha$, which is consistent with the analysis presented earlier in \eqref{a10_2d}. Our focus is to investigate the discrepancies between the local and nonlocal models with respect to the interaction radius $\delta\rightarrow 0$ in this setting.  
In Figure~\ref{figerr2dsingular}, we plot the discrete $L^2$-errors in log-log scale against $h$ for various values of small $\delta\rightarrow 0$ under $\alpha = 1.3, 1.5, 1.7$. It can be observed that for different values of $\alpha$ and $\delta$, the $L^2$-errors exhibit an approximate convergence rate of $O(h^2)$. Moreover, we plot the discrete $L^2$-errors in log-log scale against $h$ for fixed ratios $\delta/h = 0.5, 2, 4$ with different $\alpha$, see Figure~\ref{figerr2dsingularradius}. The numerical results indicate that the errors exhibit an approximate convergence rate of $O(h)$, which is consistent with the one-dimensional case as in \cite{chen2025fem}.

\begin{figure}[!ht]\hspace{-6pt} 
\subfigure[$\alpha=1.3$]{
\begin{minipage}[t]{0.32\textwidth}
\centering
\rotatebox[origin=cc]{-0}{\includegraphics[width=0.95\textwidth,height=0.92\textwidth]{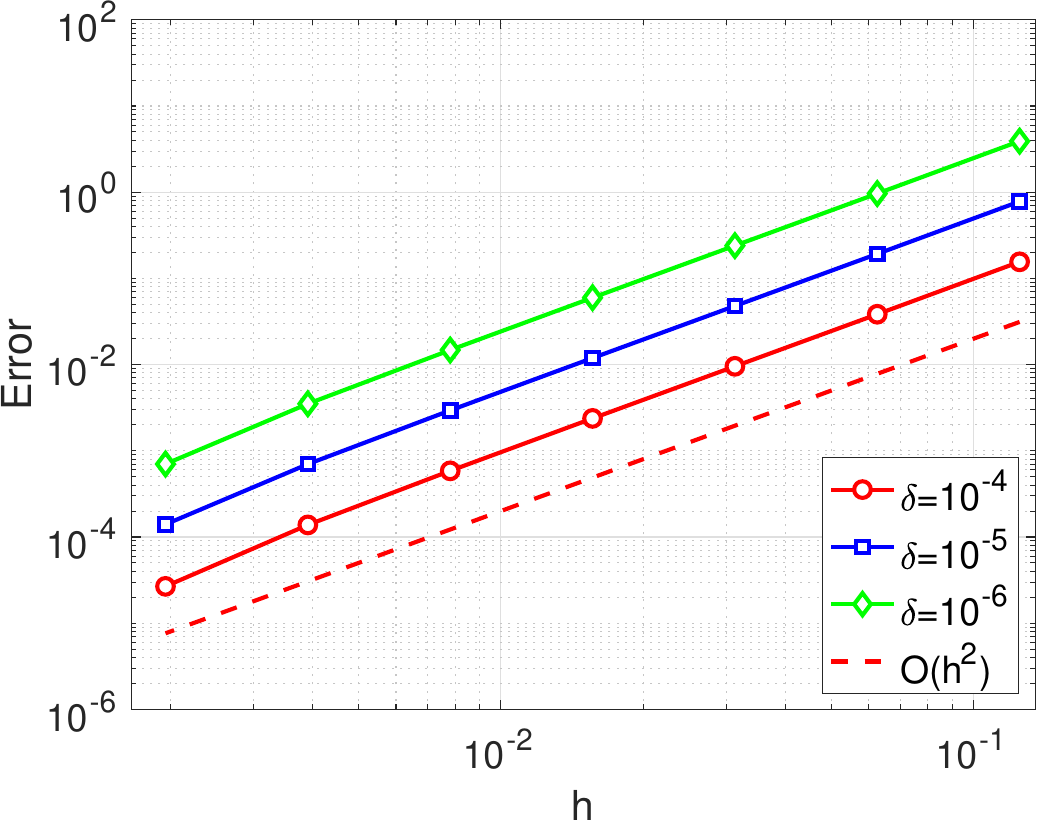}}
\end{minipage}}\hspace{-2pt}
\subfigure[$\alpha=1.5$]{
\begin{minipage}[t]{0.32\textwidth}
\centering
\rotatebox[origin=cc]{-0}{\includegraphics[width=0.95\textwidth,height=0.92\textwidth]{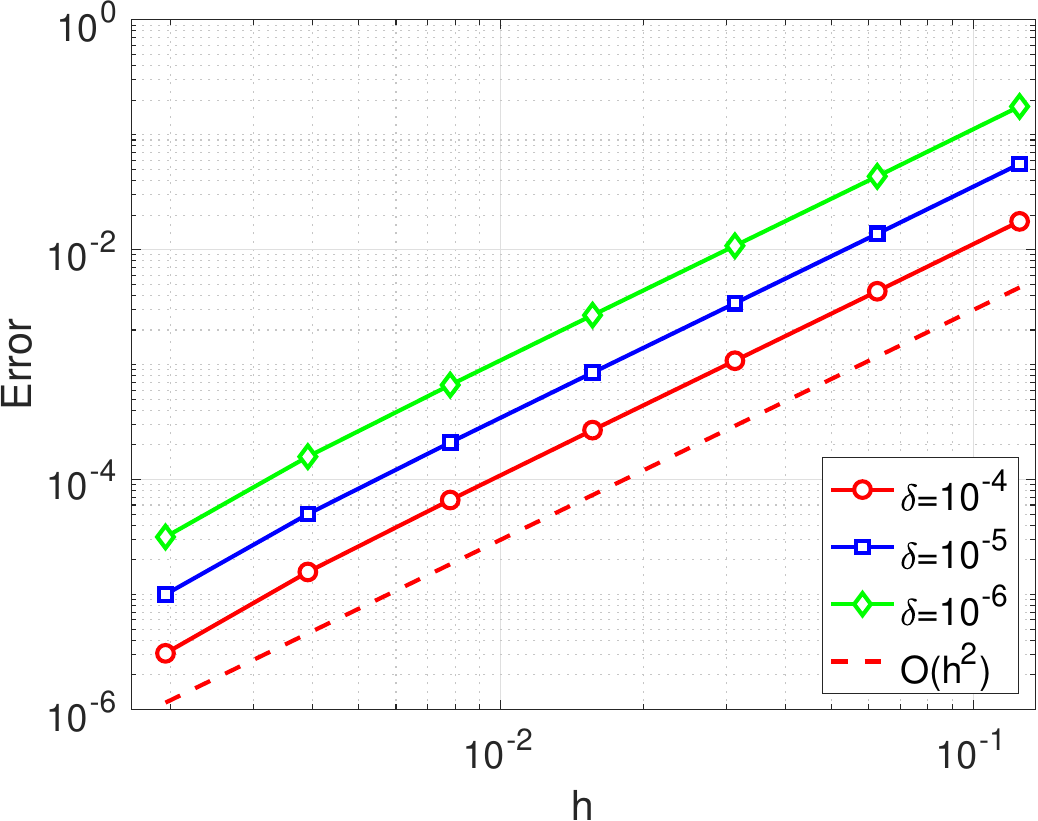}}
\end{minipage}}\hspace{-2pt}
\subfigure[$\alpha=1.7$]{
\begin{minipage}[t]{0.32\textwidth}
\centering
\rotatebox[origin=cc]{-0}{\includegraphics[width=0.95\textwidth,height=0.92\textwidth]{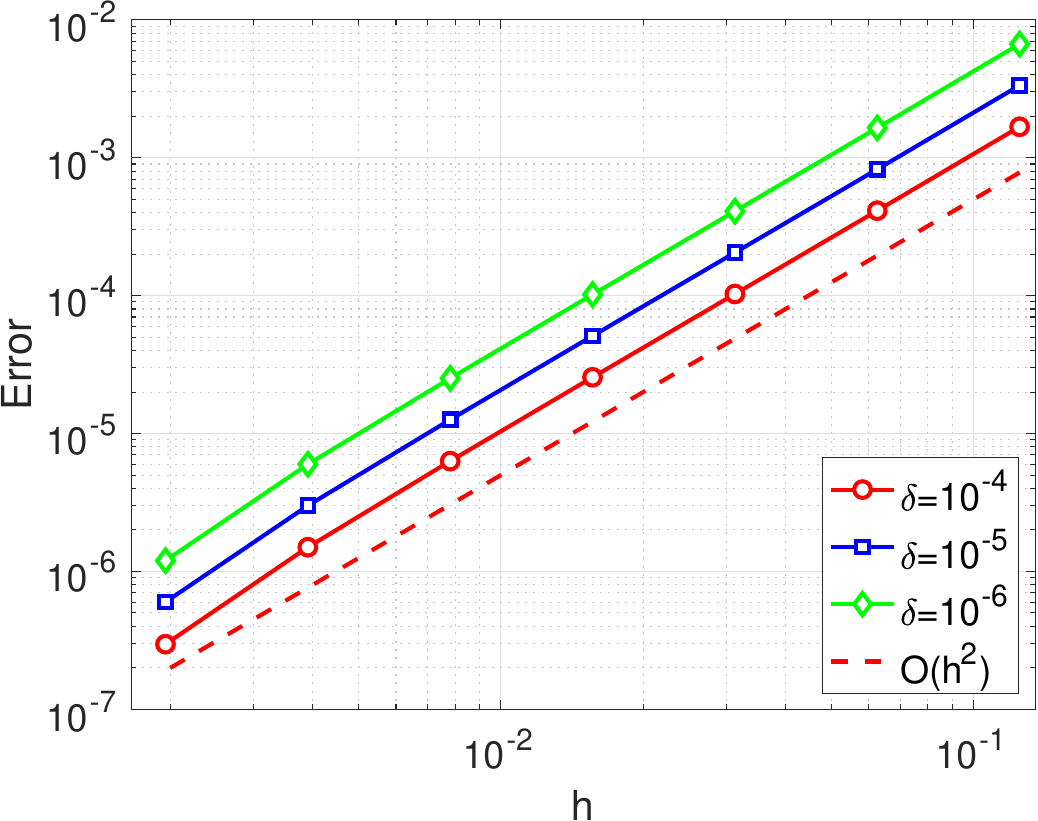}}
\end{minipage}} 
\caption
{\small The $L^2$-error and convergence order of FEM with $f (x)=1$ and fixed $\delta$ in two dimensions.  }\label{figerr2dsingular}
\end{figure}

\begin{figure}[!ht]\hspace{-6pt} 
\subfigure[$\delta=0.5h$]{
\begin{minipage}[t]{0.32\textwidth}
\centering
\rotatebox[origin=cc]{-0}{\includegraphics[width=0.95\textwidth,height=0.92\textwidth]{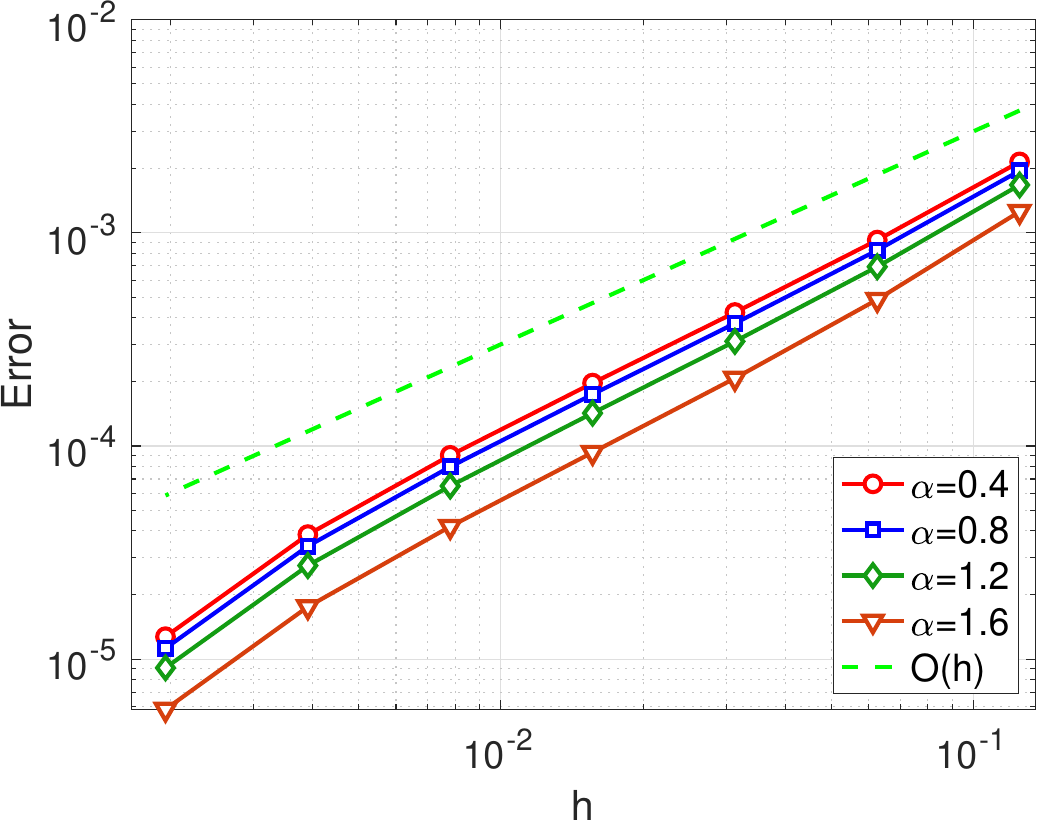}}
\end{minipage}}\hspace{-2pt}
\subfigure[$\delta=2h$]{
\begin{minipage}[t]{0.32\textwidth}
\centering
\rotatebox[origin=cc]{-0}{\includegraphics[width=0.95\textwidth,height=0.92\textwidth]{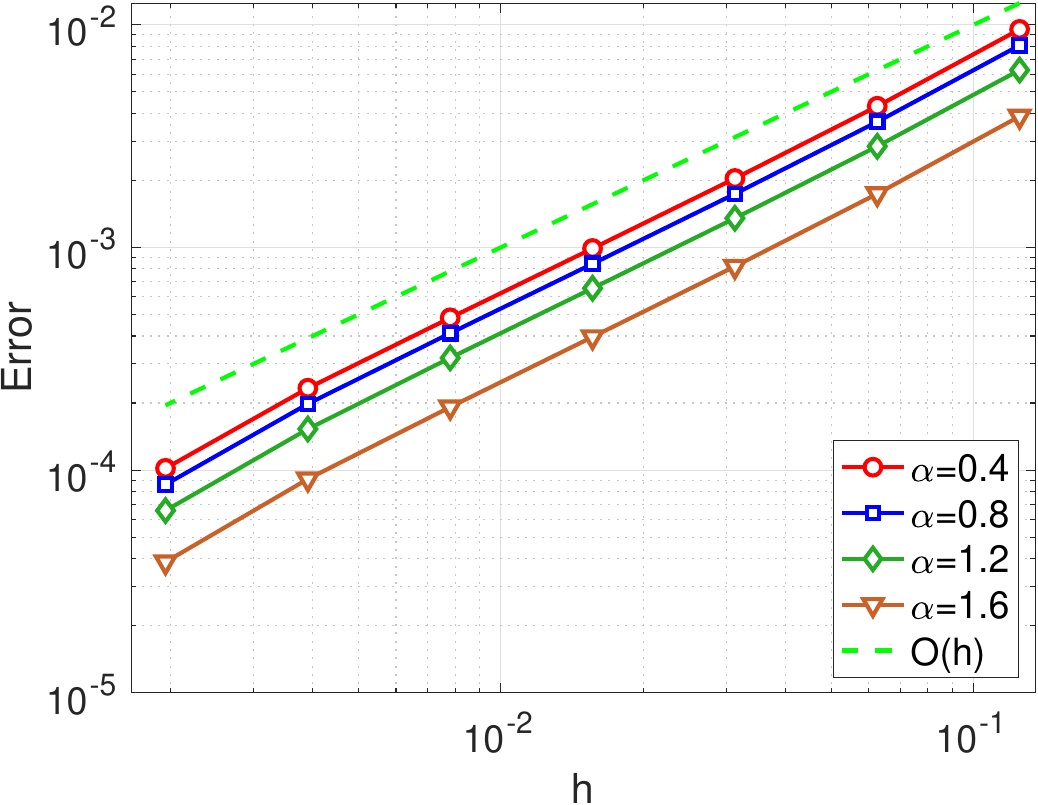}}
\end{minipage}}\hspace{-2pt}
\subfigure[$\delta=4h$]{
\begin{minipage}[t]{0.32\textwidth}
\centering
\rotatebox[origin=cc]{-0}{\includegraphics[width=0.95\textwidth,height=0.92\textwidth]{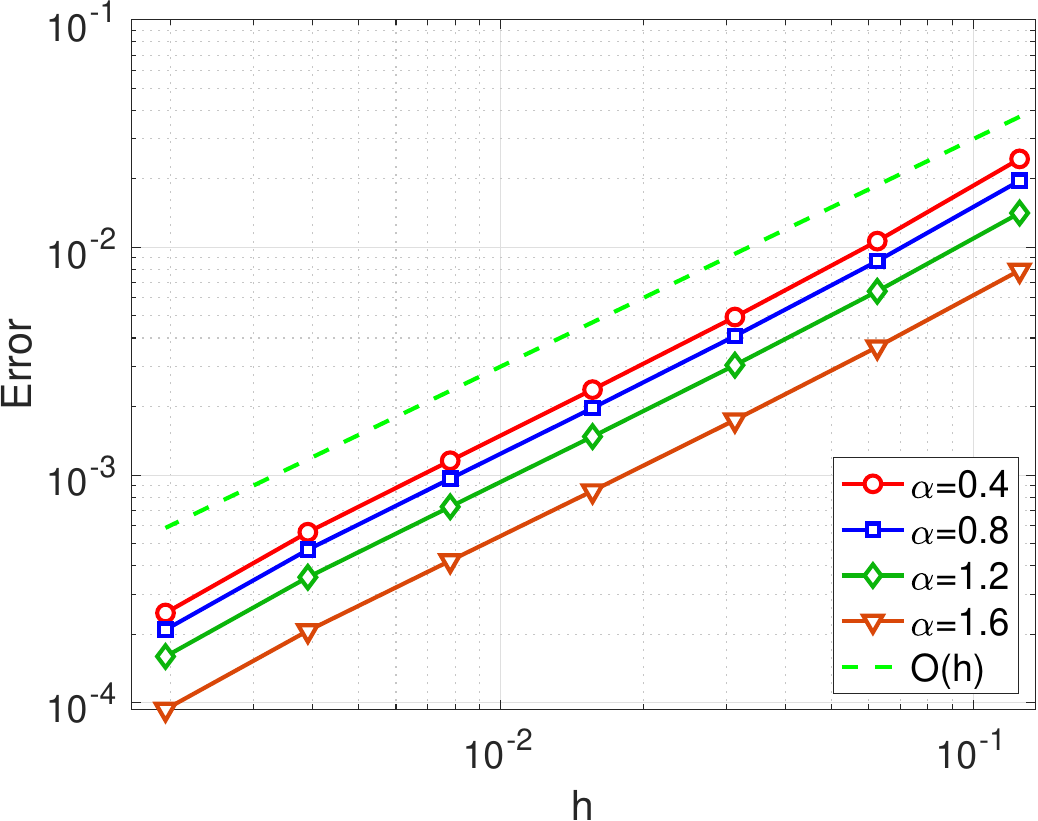}}
\end{minipage}} 
\caption
{\small The $L^2$-error and convergence order of FEM with $f (x)=1$ and fixed radius $\delta/h$ in two dimensions.  }\label{figerr2dsingularradius}
\end{figure}

 \subsection{Concluding remarks}\label{con}
In this paper, we develop a novel finite element method for solving nonlocal Laplacian equations on bounded domains in multiple dimensions. By exploiting the convolution property of B-spline basis functions, we show that the entries of the FEM stiffness matrix, which correspond to the nonlocal Laplacian with nonlocal homogeneous boundary conditions on rectangular meshes, can be explicitly represented as $d$-dimensional integrals over finite intervals. This formulation enables the accurate and flexible computation of matrix entries to any prescribed level of accuracy. As a byproduct, for the case $\delta \ge h$, we derive explicit expressions for the stiffness matrix entries corresponding to hypersingular kernels. Extensive numerical results are presented to verify the optimal accuracy and computational efficiency of the proposed method. In addition, by combining it with the grid-overlay technique proposed in \cite{huang2024,huang2025}, the method can be readily extended to problems defined on complex geometries.

\bigskip 
\noindent{\bf Acknowledgment:}\,  The authors would like to express their gratitude to Professor Weizhang Huang from University of Kansas and Dr. Jinye Shen from Southwestern University of Finance and Economics  for  insightful discussions, in particular, on exploring Go-FEM for nonlocal problems on  unstructured meshes.

\begin{appendix}
 
\section{The explicit expressions of the integrand in 2D} \label{appendixA}
 \renewcommand{\theequation}{A.\arabic{equation}}
For the convenience of computation, we provide the explicit expressions of $\hat{f}_{k_1,k_2}(r,\theta)$ with $k_1\ge k_2$, for any $(r,\theta)\in(0,1]\times(0,\pi)$  as below
 \begin{equation*}\label{f01} 
\begin{split} 
&\hat{f}_{01}(r,\theta)=-\frac{r^2}3\Big(2-\frac43r\sin\theta-\big(3+3r^2-\frac13r(4+6r^2)\sin\theta\big)\cos^2\theta
\\&\;\;+\frac{r}2\big(1-3r^2+2r^3\sin\theta\big)|\!\cos^3\theta|+r^2\big(3-2r\sin\theta\big)\cos^4\theta+\frac{r^3}2(2-3r\sin\theta)|\!\cos^5\theta|\Big), 
\\&\hat{f}_{20}(r,\theta)=-\frac{r^3}{36}\big(4-6r^2\cos^2\theta+3r^3|\!\cos\theta|\big)\sin^3\theta,
\end{split}\end{equation*}
and
\begin{equation*}\label{f11} 
\begin{split}
\hat{f}_{11}(r,\theta)=-\frac{r^2}6\begin{cases}1-\frac23r\sin\theta-(r^2-3)|\cos\theta|\sin\theta
\\[3pt]\;\;+\frac{r}3\big(9r+2(1-3r^2)\sin\theta\big)\cos^2\theta-\frac{r}3\big(2+6r^2-5r^3\sin\theta\big)|\cos^3\theta|
\\[3pt]\;\;-r^2\big(3-2r\sin\theta\big)\cos^4\theta+\frac{r^3}3(6-5r\sin\theta)|\!\cos^5\theta|,\;\;\;  \theta\in(0,\frac{\pi}2),
\\[3pt]1-\frac23r\sin\theta|+(3-r^2)\cos\theta|\sin\theta
\\[3pt]\;\;+r\big(3r+\frac13(6-2r^2)\sin\theta\big)\cos^2\theta+\frac{r}3\big(2+6r^2-3r^3\sin\theta\big)|\cos^3\theta|
\\[3pt]\;\;+r^2\big(2r\sin\theta-3\big)\cos^4\theta+r^3(r\sin\theta-2)|\cos^5\theta|, \;\;\; \theta\in(\frac{\pi}2,\pi),
\end{cases}
\end{split}\end{equation*} 
\begin{equation*}\label{f12} 
\begin{split}
\hat{f}_{21}(r,\theta)=-\frac{r^3|\sin^3\theta|}{36} \begin{cases}\big(1+3r(|\cos\theta|+r\cos^2\theta-r^2|\cos^3\theta|)\big),\;\;\;&\theta\in(0,\frac{\pi}2),\\[3pt]
\big(1+r|\cos\theta|\big)^3,\;\;\;& \theta\in(\frac{\pi}2,\pi),
\end{cases}\end{split}\end{equation*}
\begin{equation*}\label{f22} 
\begin{split}
\hat{f}_{22}(r,\theta)=\begin{cases}-\frac{1}{36} r^6\sin^3\theta\cos^3\theta,\;\;\;& \theta\in(0,\frac{\pi}2),\\[3pt]
0,\;\;\; &\theta\in(\frac{\pi}2,\pi).
\end{cases}\end{split}\end{equation*}

\end{appendix}

\bibliography{refnonlocal}

\end{document}